\documentclass[12pt]{amsart}
\usepackage{amscd,amsthm,amssymb,amsfonts,amsmath,epsfig,epic,bbm,url, latexsym, tikz}
\usepackage[arrow,matrix,tips,2cell]{xy}
\usetikzlibrary{arrows}
\usepackage{sidecap}\sidecaptionvpos{figure}{c}
\usepackage{hyperref}
\usepackage{caption} %for captionsetup
\usepackage{sidecap} %to adjust distance between sidecaption and image in SCfigure
\usepackage{float} %to be able to have figure in minipage

\theoremstyle{plain}
\newtheorem{theorem}{Theorem}
\newtheorem{lemma}[theorem]{Lemma}
\newtheorem{proposition}[theorem]{Proposition}
\newtheorem{corollary}[theorem]{Corollary}

\theoremstyle{definition}

\theoremstyle{remark}
\newtheorem{remark}[theorem]{Remark}

\newcommand{\Z}{\mathbb Z}    % Integers
\newcommand{\R}{\mathbb R}    % Reals
\newcommand{\C}{\mathbb C}    % Complexes
\newcommand{\PP}{\mathbb P}   % Projective space
\newcommand{\T}{\mathbb T}    % Tropical numbers
\newcommand{\TT}{\mathbb{T}}

\newcommand{\pp}{\mathcal{P}}
\newcommand{\cc}{\mathcal{C}}

\renewcommand{\ss}{\Sigma}

\newcommand{\suchthat}{\ : \ }
\newcommand{\<}{\langle}   %\< is not defined yet.
\renewcommand{\>}{\rangle} %\> is already defined.

\newcommand{\Log}{\operatorname{Log}}

\newcommand{\am}{{\mathcal{A}}}

\newcommand{\Cone}{\operatorname{Cone}}

\newcommand{\Star}{\operatorname{Star}}

\newcommand{\bsd}{\operatorname{bsd}}
\newcommand{\dsd}{\operatorname{dsd}}

\newcommand{\ignore}[1]{\relax}
\newcommand{\Perm}{\operatorname{Perm}}
\newcommand{\Conv}{\operatorname{Conv}}

\begin{document}

\title{Tailoring a pair of pants}

\author{Helge Ruddat and Ilia Zharkov}
\address{Johannes Gutenberg-Universit\"at Mainz \& Universit\"at Hamburg}
\email{ruddat@uni-mainz.de,\ helge.ruddat@uni-hamburg.de}
\address{Kansas State University, 138 Cardwell Hall, Manhattan, KS 66506}
\email{zharkov@ksu.edu}

\begin{abstract}
We show how to deform the map $\Log: (\C^*)^n \to \R^n$ such that the image of the complex pair of pants $P^\circ \subset {(\C^*)^n}$ is the tropical hyperplane by showing an (ambient) isotopy between $P^\circ \subset {(\C^*)^n}$ and a natural polyhedral subcomplex of the product of the two skeleta $S\times \Sigma \subset  \am \times \cc$ of the amoeba $\am$ and the coamoeba $\cc$ of $P^\circ $. This lays the groundwork for having the discriminant to be of codimension ~ 2 in topological Strominger-Yau-Zaslow torus fibrations.
\end{abstract}
\maketitle
{\hypersetup{linkbordercolor=white} %this command is to not have a little red box displayed for the link to this footnote
\footnote{H.R. was supported by DFG grant RU 1629/4-1 and the Department of Mathematics at Universit\"at Hamburg.
The research of I.Z. was partially supported by the NSF FRG grant DMS-1265228, Simons Collaboration grant A20-0125-001 and by Laboratory of Mirror Symmetry NRU HSE, RF Government grant, ag. No. 14.641.31.0001.}
}\vspace{-1cm}

\section{Introduction}
The $(n-1)$-dimensional {pair-of-pants} $P^\circ$ is 
%the complement of $n+1$ generic hyperplanes in $\C\PP^{n-1}$. 
%By an appropriate choice of coordinates we can identify $P$ with the hypersurface in $(\C^*)^n$ defined 
%%(in homogeneous coordinates) 
%by 
%$$1+z_1+\dots z_n=0,
%$$
the main building block for many problems in complex and symplectic geometries. Its projection $\am \subset \R^n$ under the coordinate-wise $\log |z|$ map is called the \emph{amoeba} and its projection $\cc \subset \T^n$ via the argument map is the \emph{coamoeba}. Both $\am$ and $\cc$ have well known skeleta: the spine $S^\circ \subset \am$, also known as the tropical hyperplane, and $\Sigma \subset \cc$, whose cover is the boundary of the $A_n$-permutahedron. 
It is convenient to compactify ${(\C^*)^n}$ (and subspaces in it) to the product $\Delta \times \T^n$, where $\Delta$ is the standard $n$-simplex. 
We introduce a natural polyhedral complex $\pp \subset \Delta \times \T^n$, the {\em ober-tropical} pair-of-pants, which is a subcomplex of $S\times \Sigma$. 
Its projection to $S$ has equidimensional fibers with the center fiber being $\Sigma$.
The main result of the paper Theorem \ref{thm:main} states that $\pp$ is (ambient) isotopic to $P$ (in the PL category).
\begin{figure}[h]
\centering
\includegraphics[height=50mm]{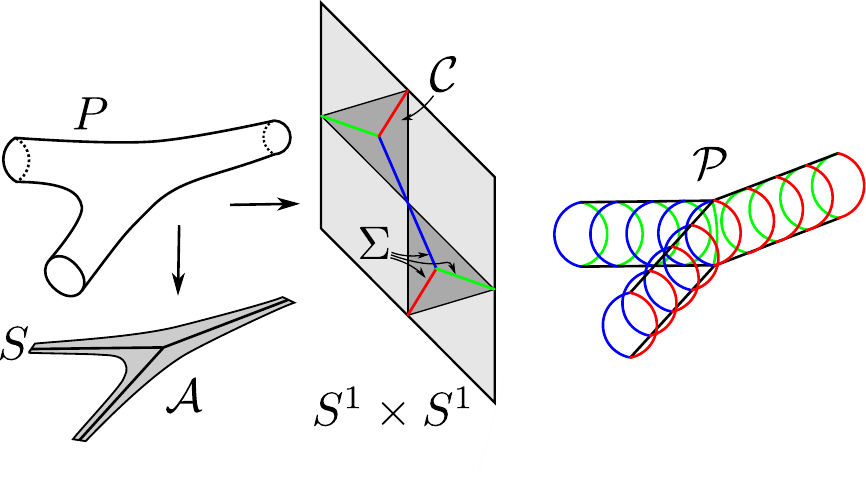}
\caption{The ober-tropical subcomplex $\pp\subset S\times \Sigma$ for $n=2$.}
\label{fig:pants2}
\end{figure}

One can compare the ober-tropical pair-of-pants to the phase tropical pair-of-pants. Both map naturally to the tropical hyperplane, but the ober-tropical version has an advantage that the fibers are half-dimensional. A homeomorphism between complex and phase tropical pairs-of-pants was established in \cite{KZ}.
The result in the present paper, namely an ambient isotopy, is considerably stronger than just a homeomorphism of the two manifolds.

One can easily visualize an isotopy for the one-dimensional pair of pants. %(cf. \cite{dim2} for the phase tropical version). 
But the simpleness of the $n=2$ case is deceptive. To get a hint of how complicated the isotopy problem becomes in higher dimensions, consider a long pentagon (a point in the 3-dimensional pair of pants, $n=4$).
\begin{figure}[h]
\centering
\includegraphics[height=20mm]{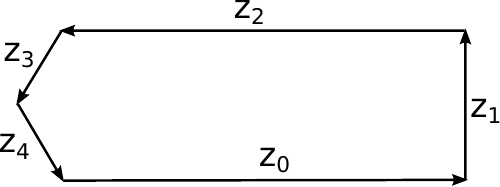}
\caption{The tall house overturned.}
\label{fig:pentagon}
\end{figure}
This pentagon maps to a point $y$ in the face  $S_{02}\colon |z_0|=|z_2|\gg |z_i|,\ i\ne 0,2$ in the spine $S$ of the amoeba. On the other hand it maps to a point $s$ in the face $\Sigma_{\<1,2340\>}$ of the skeleton of the coamoeba (the average of the four arguments of $z_0, z_2, z_3, z_4$ is opposite to the argument of~$z_1$). The problem is that, roughly speaking, the two acutest angles do not separate the two longest sides of the pentagon, so that $y\in S_{02}$ is  ``very far'' from any legitimate strata in the fiber of $s\in \Sigma_{\<1,2340\>}$. (This never happens for triangles because of the larger side lying against larger angle property). So the isotopy cannot be a ``small deformation''. 

Instead of trying to build an isotopy explicitly (which is an interesting but, perhaps, a rather difficult problem) we build regular cell decompositions of both pairs and show that they are homeomorphic. The cell structures respect the natural stratification of $\Delta \times \T^n$, so the homeomorphisms will glue well at the boundary. Thus with a little bit of effort the isotopy can be extended to any general affine hypersurface.

The isotopy we provide may be applied to several questions in mirror symmetry. One application we have in mind is the following. Given an integral affine manifold with simple singularities \cite{logmirror1} we want to build a topological Strominger-Yau-Zaslow fibration \cite{SYZ} with discriminant in codimension~2 (rather than codimension one). So far the only examples are the K3 (with discriminant consisting of 24 points in $S^2$) and the quintic 3-fold \cite{gross}, cf.~\cite{cbm,EvansMauri}. 
In a general case, the idea is roughly to modify the local models of the fibration $\{w=f\}\subset X_\Sigma\times (\C^*)^{n-k}$.  Here $X_\Sigma$ is an affine toric variety with $\Sigma$ a cone over some $k$-dimensional simplex $\Delta_1$, $w\colon X_\Sigma \to \C$  is a natural toric map, and $f\colon (\C^*)^{n-k} \to \C$ is a Laurent polynomial with a prescribed Newton polytope, a simplex~$\Delta_2$. The local model has the structure of a $T^k$-bundle over $\R^k\times (\C^*)^{n-k}$ with fibers collapsing over $\mathcal H_1 \times  H_2$ where $\mathcal H_1$ is the codimension one skeleton of the normal  fan to $\Delta_1$, and $H_2= \{f=0\}$ is a hypersurface in $(\C^*)^{n-k}$. One then further projects to the $\R^{n-k}$-factor by the $\Log$ map in $(\C^*)^{n-k}$. To have the $T^n$-fibration $\{w=f\}\to \R^n$ with discriminant in codimension 2, we can use our isotopy to replace the pair $H_2 \subset (\C^*)^{n-k}$ with a pair $ \mathcal H_2 \subset (\C^*)^{n-k}$ such that now $\mathcal H_2$ is mapped to the {\em spine} of the amoeba of $H_2$.
The details are in our forthcoming paper~\cite{RZ}.

%A direct consequence of the isotopy in our main Theorem~\ref{thm:main} is a refinement of Mikhalkin's pair-of-pants decomposition theorem \cite{PP}:
%\begin{corollary} 
%Let $H \subset (\C^*)^n$ be a hypersurface that is close to the tropical limit with smooth tropicalization $S\subset \R^n$. 
%There exists an isotopy of the $\log|z|$ map $(\C^*)^n\to\R^n$ so that the resulting map has $H$ fiber over $S$ in $(n-1)$-dimensional fibers. 
%The fiber over a point in the relative interior of a $k$-cell of $S$ is homeomorphic to $(S^1)^k\times \Sigma^{n-k-1}$ where $\Sigma^{d}$ denotes the skeleton of the coamoeba of the $d$-dimensional pair of pants (see \S\ref{sec:coamoeba}).
%\end{corollary}
%\begin{proof} If $H$ is a pair of pants, the corollary is our main result. 
%The more general situation reduces to the pair of pants case by assumption of the tropicalization to be smooth so that we can either apply Mikhalkin's decomposition construction \cite{PP} or a degeneration of the ambient $(\C^*)^n$ so that $H$ breaks into pairs of pants similar to \cite{RSTZ}. We use in either case that the isotopy of the pair of pants respects the boundary in a well-behaved way.
%\end{proof}

\subsection*{Thanks} The authors would like to thank the organizers of the Tropical workshop at Oberwolfach in May 2019 where the second author came up with the new object $\pp$ for the first time (hence the name). Our gratitude for hospitality additionally goes to the Mittag-Leffler institute, Institute of the Mathematical Sciences of the Americas and MATRIX in Creswick. We thank Dave Auckley for numerous discussions about the topology of ball pairs and we thank the referee for careful reading.

\section{Geometry of the complex pair of pants}

\subsection{Notations} 
We set $\hat n:= \{0,\dots,n\}$.
We will think of $(\C^*)^{n} \cong (\C^*)^{n+1}/\C^*$ as the product $\Delta^\circ \times \T^n$ where $\Delta^\circ$ is the interior of the $n$-simplex 
$$\Delta:=\left\{(x_0, \dots, x_n) \in \R^{n+1} \suchthat x_i\ge 0, \ \sum x_i =1 \right\}
$$ 
and $\T^n := (\R/2\pi \Z)^{n+1}/(\R/2\pi \Z) $ is the $n$-torus with homogeneous coordinates $[\theta_0, \dots, \theta_n]$. It is more natural to work with closed spaces, especially for the gluing purposes, so we will compactify $(\C^*)^{n}$ to  $\Delta \times \T^n$ and all subspaces in $(\C^*)^{n}$ by taking their closures in $\Delta \times \T^n$. We will denote by $\bsd\Delta$ the first barycentric subdivision of $\Delta$. We will also consider the\\[1mm]
\noindent
\begin{minipage}[b]{0.50\textwidth} 
{\bf dualizing} subdivision
$\dsd \Delta$ which is a coarsening of $\bsd\Delta$ by combining all simplices in $\bsd\Delta$ from a single interval $[I,J]$ together. 
That is, a cell $\Delta_{IJ}$ in $\dsd \Delta$ is
$$\Delta_{IJ}:= \Conv\{ \hat \Delta_K \suchthat I\subseteq K \subseteq J \},
$$
where $\hat \Delta_K$ stands for the barycenter of $\Delta_K$. 
\end{minipage}\qquad
\begin{minipage}[t]{0.4\textwidth}
\raisebox{3mm}{
\includegraphics[width=1\textwidth]{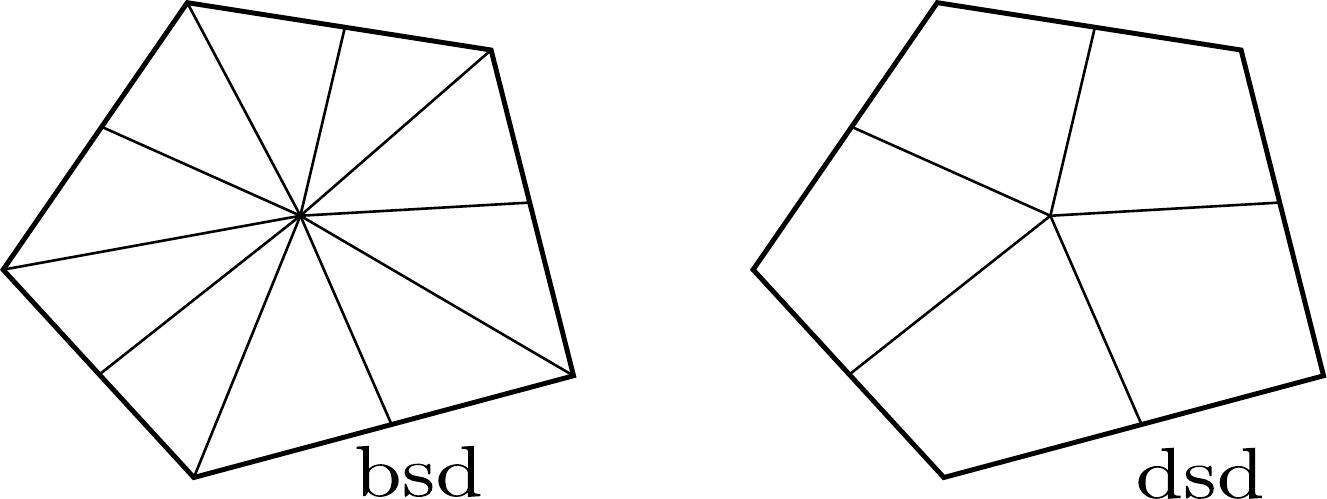} 
}
\end{minipage}

\noindent
Another cell $\Delta_{I'J'}$ is a face of $\Delta_{IJ}$ if and only if $I\subseteq I'\subseteq J'\subseteq J$.
The dualizing subdivision can be applied to any convex polytope $Q$ where it interpolates between the cone structure of the normal fan in the interior of $Q$ and the original face stratification at the boundary of $Q$. For a general polytope $Q$ the cells in $\dsd Q$ need not be polytopes anymore, but $\dsd Q$ is still a regular CW-complex. 
%The subdivision $\dsd Q^\vee$  of the dual polytope $Q^\vee$ is isomorphic to $\dsd Q$ as CW-complexes.

The {\bf hypersimplex} $\Delta^n(2) \subset \Delta^n$ is obtained from the ordinary simplex by cutting the corners half-way. That is, 
$$\Delta^n(2):=\{(x_0,\dots,x_n) \in \R^{n+1}  \suchthat \sum x_i=2\pi  \text{ and } 0\leq x_i \leq \pi \}.
$$
We will use $2 \pi=1$ for the amoeba and $2\pi=6.28\dots$ for the coamoeba.

The {\bf cyclic polytope} $C_d(r)$ is the convex hull of the points $x(t_1), \dots, x(t_r)$ in $\R^d$, where $x(t)=(t,t^2,\dots, t^d)$ and $t_1 < \dots <t_r$ are real numbers.

\subsection{A cell decomposition of the complex pair of pants} 
The $(n-1)$-dimensional {\bf pair-of-pants} $P^\circ$ is the complement of $n+1$ generic hyperplanes in $\C\PP^{n-1}$. 
By an appropriate choice of coordinates we can identify $P^\circ$ with the affine hypersurface in $(\C^*)^{n+1}/\C^*$ given in homogeneous coordinates by 
$$z_0+z_1+\dots + z_n=0.
$$ 
We define the {\bf compactified pair-of-pants} $P$ to be the closure of $P^\circ$ in $\Delta \times \TT^n$ via the map 
$$\mu_1\times \mu_2\colon (\C^*)^{n+1}/\C^* \to \Delta \times \TT^n , \quad 
[z_0,\dots ,z_n] \mapsto \left(\frac{|z_0|}{\sum |z_i |},\dots, \frac{|z_n|}{\sum |z_i |}; [\arg z_0, \dots, \arg z_n] \right)
$$
where we abused notation when writing the quotient of $\TT^{n+1}$ modulo the diagonal $\TT^1$ simply by $\TT^{n}$ (with coordinates $\theta_1-\theta_0,...,\theta_n-\theta_0$).
The closure $P$ is a manifold with boundary, and it can be thought of as a real oriented blow-up of $\C\PP^{n-1}$ along its intersections with the coordinate hyperplanes in $\C\PP^n$.

\begin{SCfigure}[][h]
\centering
\includegraphics[width=0.5\textwidth]{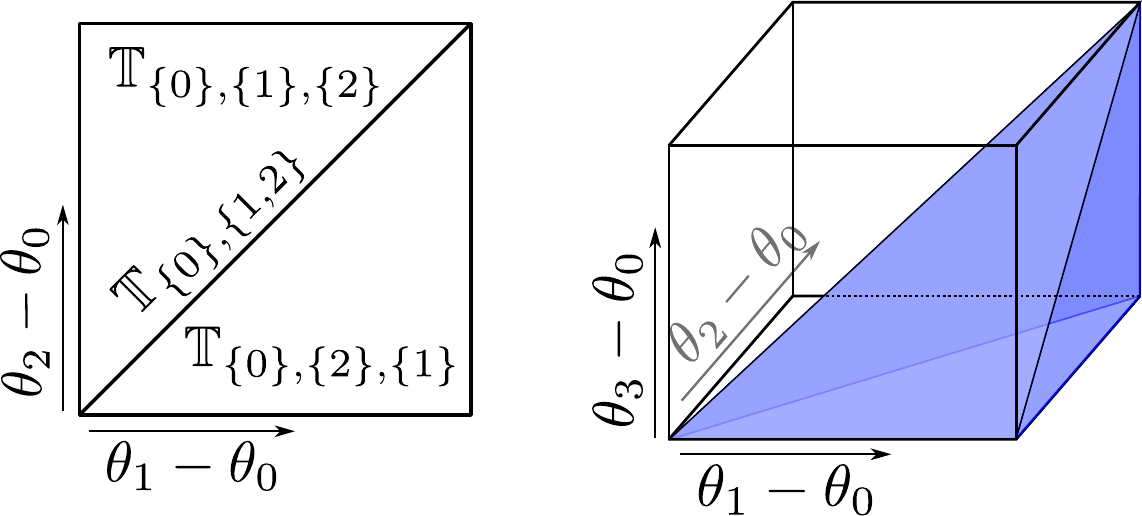}
\caption{Stratification of the torus $\mathbb{T}^n$ for $n=2$ and $n=3$. 
The right only shows 
$\mathbb{T}^3_{\{0\},\{3\},\{2\},\{1\}}$ 
%$\mathbb{T}^3_{\{0\},\{1\},\{3\},\{2\}}$ 
%$\mathbb{T}^3_{\{0\},\{2\},\{3\},\{1\}}$ 
in blue and there exist five further maximal strata. Only zero-dimensional strata are closed.}
\label{fig:torus-strat}
\end{SCfigure}
Next we review a natural stratification $\{P_{\sigma,J}\}$ of $P$ from \cite{KZ} induced from a product stratification of $\Delta \times \TT^n$. The stratification $\{\Delta_J\}$ of the simplex $\Delta$ is given by its face lattice, namely by the non-empty subsets $J\subseteq \hat n$. The set of the hyperplanes $\theta_i=\theta_j,\ i,j\in \hat n$, stratifies the torus $\T^n$ by cyclic orderings of the points $\theta_0, \dots, \theta_n$ on the circle. The strata $\T^n_\sigma$ are labeled by {\bf cyclic partitions} $\sigma=\< I_1, \dots, I_k\> $ of the set $\hat n$, that is, $\hat n= I_1 \sqcup \dots \sqcup I_k$ and the sets $I_1, \dots, I_k$ are cyclically ordered. The elements within each $I_i$ are not ordered. We call the $I_i$ the {\bf parts} of $\sigma$ and we set $|\sigma|:=k$. If all parts are 1-element sets then we will write $\sigma=\<i_0, \dots, i_n\>$. 

We can view points in $P_{\sigma,J} $ as (degenerate) convex polygons in the plane defined up to rigid motions and scaling. The edges represent the complex numbers $z_0, \dots, z_n$ ordered counterclockwise according to $\sigma$. The edges within each $I_s$ are parallel and not ordered. 
The edges of the polygon which are not in $J$ have zero length but their directions are still recorded.

Each $\T^n_\sigma$ can be thought of as the interior of the simplex 
\begin{equation}\label{eq:lift}
\Delta_\sigma:=\left\{(\alpha_1, \dots, \alpha_k) \in \R^{k} \suchthat \alpha_i\ge 0, \ \sum \alpha_i =2\pi \right\}
\end{equation}
The coordinates $\alpha_i$ play the r\^ole of exterior angles of the convex polygons. The precise relation between $\alpha$'s and the original arguments $\theta$'s is as follows: $\theta_i=\theta_j$ for $i,j\in I_s$ and $\theta_i+\alpha_s=\theta_j$ for $i\in I_s ,j\in I_{s+1}$. Here we assume the periodic indexing $I_{s+k}=I_s$.

The inclusion of closed strata $P_{\sigma',J'}\subseteq P_{\sigma,J}$ gives a partial order among the pairs: $(\sigma', J') \preceq (\sigma, J)$ if $\sigma$ is a refinement of $\sigma'$ (we write $\sigma'\preceq \sigma$) and $J'\subseteq J$. To simplify notations we will drop the index $J$ from the subscript if $J=\hat n$. 

We have a natural surjection $\Delta_\sigma\rightarrow\overline{\T^n_{\sigma}}$ which is bijective away from the vertices (all the vertices map to $0\in\T^n$).
Since $P_{\sigma,J}$ does not 
have any points lying over the vertices\\[.6mm]
\noindent
\begin{minipage}[b]{0.50\textwidth} 
of $\Delta_\sigma$, we may view $P_{\sigma,J}$ to be sitting as
$$P_{\sigma,J} \subseteq \Delta_J \times \Delta_\sigma.$$
We set $\sigma_0=\<0,1,\dots,n\>$ and $\Delta_0:=\Delta_{\sigma_0}$. For $J=\hat n$ and $\sigma=\sigma_0$ we denote the corresponding maximal stratum of $P$ by $P_0\subseteq \Delta\times\Delta_0$.
All maximal strata are isomorphic to~$P_0$.

\quad
We say that $\sigma$ {\bf divides} $J$, and write $\sigma | J$, if
\end{minipage}\qquad
\begin{minipage}[t]{0.4\textwidth}
\includegraphics[width=1\textwidth]{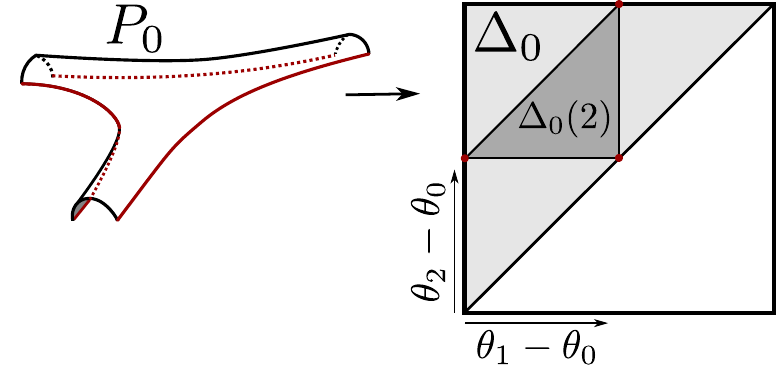}
\end{minipage}

%it will be convenient to lift the torus stratum $\T^n_{\sigma_0}$ to the simplex $\Delta_0$. That is 
%That does not affect the pair of pants because $\theta_0=\dots=\theta_n=0$ point in $\T^n$ does not have a preimage in $P$.
%We will denote the corresponding stratum by $P_0 \subseteq \Delta \times \Delta_0$. All maximal strata are isomorphic to~$P_0$.

\noindent
$J$ contains elements in at least two of the subsets $I_1, \dots, I_{k}$ of $\sigma=\<I_1, \dots, I_{k}\>$.
The face poset of a stratum $P_{\sigma,J}$ consists of pairs $(\sigma', J')$ such that $\sigma'\preceq \sigma$,   $J'\subseteq J$ and $\sigma'$ divides $J'$. This, in particular, means that $|\sigma'|\ge 2$ and $|J'|\ge 2$. If $\sigma$ does not divide $J$ the stratum $P_{\sigma,J}$ is empty. The dimension of $P_{\sigma,J}$ is $|\sigma| + |J| -4$. 

To represent $(\sigma,J)$ we will use the following {\bf graphical code}. 
Once the maximal partition $\sigma_0=\langle i_0,...,i_n\rangle$ is fixed, we may think of $\sigma_0$ as a cyclic labelling of the edges of the $(n+1)$-gon or equivalently a labelling of the $n+1$ arcs on a circle separated by $n+1$ vertices.
Now every coarsening $\sigma\preceq \sigma_0$ can be represented as a subset $V$ of the vertices, which we call {\bf vertices} of $\sigma$ that separate parts in $\sigma$.
%$I_j$ and $I_{j+1}$ for some $j$ (in cyclic order). 
Furthermore, we view $J$ as a subset of the edges (or arcs). 
Consequently, we may represent $(\sigma,J)$ as a circle with $n+1$ vertices and $n+1$ labeled arcs, out of which the arcs in $J$ and the vertices in $V$ have been marked. 
We graphically indicate an arc marking by connecting the adjacent vertices of the arc by a straight line, see Figure~\ref{fig:interlacing}.

%The direct connection of a graphical code to the cell $P_{\sigma,J}$ is: a vertex is marked if and only if the adjacent variables have different phase in the interior of $P_{\sigma,J}$; an arc is marked if and only if the associated variable has positive absolute value in the interior of $P_{\sigma,J}$.
%
The $\sigma$-divides-$J$ property now translates into saying that the pair $(V,J)$ be {\bf interlacing}, which means that not all edges in $J$ are lying in an arc between two elements of $V$. 
We call a pair $(V,J)$ {\bf maximally interlacing} if it is interlacing and either every part of $\sigma$ has an element from $J$ (if $|\sigma|\le |J|$) or every part of $\sigma$ contains at most one element of $J$ (if $|\sigma|\ge |J|$).
The lowest strata of $P_0$ are the (maximal) interlacing pairs $(V,E)$ where $V$ is a pair of vertices and $E$ is a pair of edges separated by $V$.
\begin{figure}[h]
\centering
\includegraphics[width=.5\textwidth]{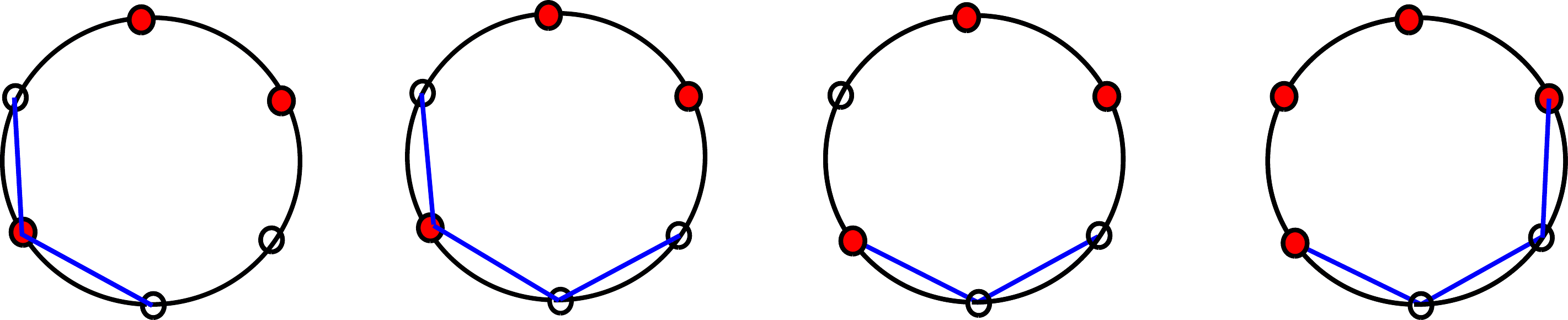}
\caption{For $n+1=6$, examples of graphical code referencing subsets $(V,J)$ that are \emph{maximal} and \emph{non-maximal interlacing}, \emph{non-interlacing}, and \emph{maximal non-interlacing}. The vertex subsets are indicated by red coloring and the edge markings by blue segments.}
\label{fig:interlacing}
\end{figure}
A pair $(V,E)$ is {\bf maximal non-interlacing} if adding an extra element to either $V$ or $E$ will make it interlacing.

The following observation (cf.~Conjecture 1 in \cite{KZ}) was suggested by Stanley \cite{Stanley}.

\begin{proposition}\label{prop:stanley}
The face lattice of $P_0$ is dual to the face lattice of the cyclic polytope $C_{2n-2}(2n+2)$.
\end{proposition}
\begin{proof}
%It is clear that for any pair $(\sigma, J) \preceq (\sigma_0, \hat n)$ the interval $[(\sigma, J), (\sigma_0, \hat n)]$ is Boolean, which means that the polytope would have to be simple. 
If we label the edges $z_i$ of the polygon by $t_{2i}$ and the vertices between $z_i$ and $z_{i+1}$ by $t_{2i+1}$ then the interlacing condition for a minimal pair $(V,E)$ is equivalent to that there are even number of $t_j$'s between any two consecutive elements of $\{t_{e_1},t_{v_1}, t_{e_2}, t_{v_2}\}$. This confirms the Gale evenness condition (see, e.g. \cite[Theorem 0.7]{ziegler}), for the facets (our vertices) of the cyclic polytope. 
\end{proof}

The lower dimensional strata are isomorphic to the faces of $P_0$ and hence every lower interval $\mathcal W_{\preceq(\sigma,J)} :=\{(\sigma',J')\in \mathcal W \suchthat (\sigma',J') \preceq (\sigma, J)\}$ is the face lattice of a polytope $C^\vee_{\sigma,J}$. In particular, the facets of $P_0$ are dual to the cyclic polytope $C_{2n-3}(2n+1)$. 

\begin{remark}
The appearance of the rational normal curve in the picture seems very intriguing.
\end{remark}

\subsection{The amoeba and its skeleton} \label{sec:amoeba}
The image $\am:=\mu_1(P)\subset \Delta$ is called the (compactified) {\bf amoeba} of the hypersurface $P$ \cite{GKZ}. It is easy to see that $\am$ is the hypersimplex $\Delta(2)$, since the only restrictions on lengths of the $z_i$ are given by the triangle inequalities. That is, if we normalize the perimeter to be 1, then $\am$ is cut out by the inequalities $|z_i|\le1/2$.

The {\bf skeleton} $S$ of $\am$ (also known as the spine or the tropical hyperplane) is a polyhedral subcomplex of $\dsd\Delta$ defined as 
$$S=\{(x_0, \dots ,x_n) \in \Delta \suchthat x_i=x_j\ge x_k \text{ for some } i\ne j \text{ and all } k\ne i,j\}.
$$ 
The faces $S_{IJ}$ of $S$ are parameterized by pairs of subsets $I\subseteq J\subseteq \hat n$, such that $|I|\ge 2$. Namely, $S_{IJ}$ is defined by $x_i=x_{i'} \ge x_k$ for all $i,i' \in I , k\not\in I$ and $x_j=0, j\not\in J$.
A~well-known observation is the following, e.g.~by \cite{PR04}, Theorem 1\,(iii).
\begin{proposition}
The amoeba $\am$ deformation retracts to its skeleton $S$.
\end{proposition}

%A skeleton $\ss$ of the complex pair-of-pants $P^{n-1}$ plays an important role in certain calculations in mirror symmetry (cf. \cite{Sheridan}, \cite{FuUe14}). It sits very naturally in the coamoeba $\cc$ of $P^{n-1}$. Here we will describe a particular lift of $\ss$ to the pair-of-pants. 
%
%We give a description of $\ss$ in terms of a Voronoi tiling, which might be somewhat different from the standard one as the convex hull of vertices. A good reference for Voronoi tilings is \cite{Erdahl}.
%

\subsection{The permutahedron and the zonotope}
Before we describe the coamoeba and its skeleton we review some notations and basic facts from the $A_n$-root system terminology. 
We write $\R^{n+1}/\R$ for the quotient space $\R^{n+1}/\R(1,...,1)$.
Let $e_0, \dots, e_n$ be a basis of the Euclidean space $(\R^{n+1})^*$. The {\bf roots} in $\Lambda_\R :=(\R^{n+1}/\R)^*
%\subset (\R^{n+1})^*
$ are $\alpha_{ij}=e_i-e_j$. The inner product on $\Lambda_\R$ is induced from $(\R^{n+1})^*$. Over $\Z$ the roots generate the {\bf root lattice} $\Lambda \subset\Lambda_\R$. 

The {\bf weight lattice} $W\subset \Lambda_\R$ is the dual lattice to $\Lambda$ with respect to the inner product. For any partition of $\hat n$ into two non-empty subsets $\{I_-,I_+\}$, we define the {\bf fundamental weight}
\begin{equation}\label{eq:fundamental}
w_{I_-,I_+} = \frac{r}{n+1} \sum_{j\in I_+} e_j - \frac{s}{n+1} \sum_{i\in I_-} e_i \ \in W ,
\end{equation}
where $r=|I_-|$ and $ s=|I_+|=n+1-r$.
Its  length squared is 
$$|w_{I_-,I_+}|^2= \frac{s r^2 + r s^2}{(n+1)^2} = \frac{rs}{n+1}.
$$
There are $2^{n+1}-2$ fundamental weights.
\begin{figure}[h]
\centering
\includegraphics[width=.9\textwidth]{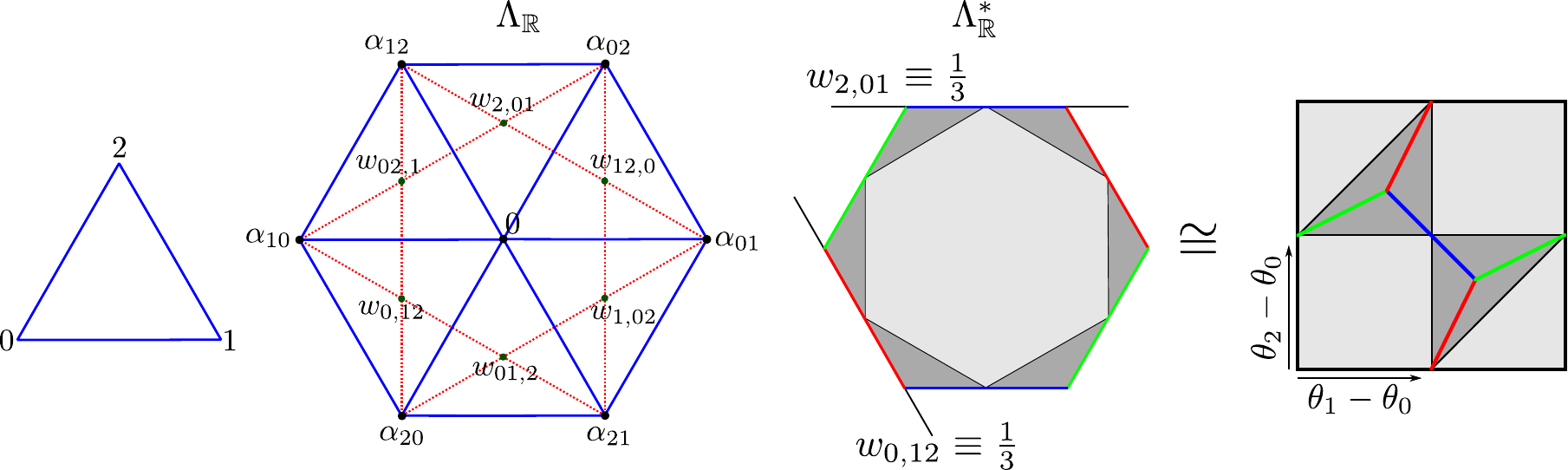}
\caption{For $n=2$: a choice of homogeneous coordinates of the triangle, the corresponding roots and fundamental weights, the permutahedron with zonotope inside and how these give rise to the coamoeba and its skeleton}
\label{fig:permn2}
\end{figure}
 
The quadratic form $Q(w)=\pi |w|^2$ on $W$ provides the {\bf Delaunay triangulation} of $\Lambda_\R$ by looking at the convex upper hull of its graph. This triangulation is $W$-periodic. 
The edge vectors originating at 0 are the fundamental weights. The maximal simplices incident to 0 correspond to permutations $\{i_0, \dots, i_n\}$ of the set $\hat n$. Such a permutation determines a choice of positive and simple roots in~$\Lambda$. Then the convex hull of the corresponding set of $n$  {\bf dominant} fundamental weights
\begin{equation}\label{eq:fund_alcove}
\frac1{n+1} (-1, \dots, -1, n), \quad  \frac1{n+1} (-2, \dots, -2, n-1, n-1), \quad \dots \quad, \frac1{n+1} (-n, 1, \dots, 1)
\end{equation}
(written in the basis $e_{i_0}, \dots, e_{i_n}$) together with the origin gives a maximal Delaunay simplex. The cone they generate is knows as the Weyl alcove.

More generally, the faces in the Delaunay triangulation (up to translations by $W$) are labeled by ordered partitions $\bar \sigma = \{I_1, \dots , I_k\}$ of $\hat n$. The single vertex (up to translations by $W$) corresponds to the 1-partition $\bar \sigma = \{\hat n\}$.

We now consider the dual space $\Lambda_\R^*=\R^{n+1}/\R$. The basis $e_0, \dots, e_n$ defines homogeneous coordinates $[x_0, \dots, x_n]$ in $\Lambda_\R^*$.  The {\bf coweight lattice} $\Lambda^*$ which is dual to $\Lambda$ is naturally identified with the lattice $\Z^{n+1}/\Z \subset \R^{n+1}/\R$. 

The $A_n$-type Delaunay triangulation of $\Lambda_\R$ is a {\bf dicing} with respect to $\Lambda$. That is, it is given by a totally unimodular system of families of parallel hyperplanes (cf.~\cite{Erdahl}). An equivalent statement is that the dual decomposition of $\Lambda_\R^*=\R^{n+1}/\R$, known as the {Voronoi tiling}, is by zonotopes. This tiling is $2 \pi \Lambda^*$-periodic.
Explicitly,  the Voronoi tiles are the domains of linearity of the Legendre transform of $Q(w)$ -- the convex quasi-periodic PL function on $\R^{n+1}/\R$:
$$\Theta(u)=\max_{w\in W} \{\<w,u\> - \pi |w|^2\}.
$$

Now let's go back to the skeleton.
We denote by $\Perm$ the central maximal Voronoi tile (the one which contains 0 in its interior). Other maximal tiles are translations of $\Perm$ by elements in the lattice $2\pi \Lambda^*$. In homogeneous coordinates $[x_0, \dots, x_n]$ the central tile is defined by the following inequalities:
\begin{equation}\label{eq:perm_facets}
-\frac{s}{n+1} \sum_{i\in I_-} x_i +  \frac{r}{n+1}\sum_{j\in I_+} x_j \le \frac{rs}{n+1} \pi,
\end{equation}
one inequality per each fundamental weight $w_{I_-,I_+}$. 

The vertices of $\Perm$ are in one-to-one correspondence with the maximal simplices of the Delaunay triangulation incident to 0. More precisely, the vertex of $\Perm$ corresponding to a permutation $\bar \sigma=\{i_0, \dots, i_n\}$ is given in homogeneous coordinates $[x_{i_0}, \dots, x_{i_n}]$ by
\begin{equation}\label{eq:rho}
\rho_{\bar \sigma}^\vee=\frac{\pi}{n+1}[1, 3, \dots, 2n-1, 2n+1].
\end{equation} 
Indeed, we evaluate $\rho_{\bar \sigma}^\vee$ on the set of the dominants weights \eqref{eq:fund_alcove}. The calculation for a weight 
$$w_{I_-,I_+}=(\underbrace {-\frac{s}{n+1}, \dots, -\frac{s}{n+1}}_r, \underbrace{\frac{r}{n+1}, \dots, \frac{r}{n+1}}_s)
$$
becomes more transparent if we represent $\rho_{\bar\sigma}^\vee$ by
$$\rho_{\bar\sigma}^\vee=\frac\pi{n+1}[-2r+1, -2r+3, \dots , -1, 1, \dots, 2s-1].
$$ 
Then 
$$\<\rho_{\bar\sigma}^\vee, w_{I_-,I_+}\> = \frac\pi{(n+1)^2}(r^2 s + s^2 r) = \pi\frac{rs}{n+1}=\pi Q(w_{I_-,I_+}). 
$$ 

The vertices of $\Perm$ are permuted under the action of the symmetric group action, so $\Perm$ is indeed a permutahedron. But it is a very special one: it tiles the space.
 
%\begin{figure}
%\centering
%\includegraphics[height=50mm]{perm_zono4.pdf}
%\caption{Permutahedral tiling of $\R^3$ (in red). The grey are the zonotopes $Z$.}
%\label{fig:zono4}
%\end{figure}

\begin{figure}[h]
\centering
\includegraphics[width=.25\textwidth]{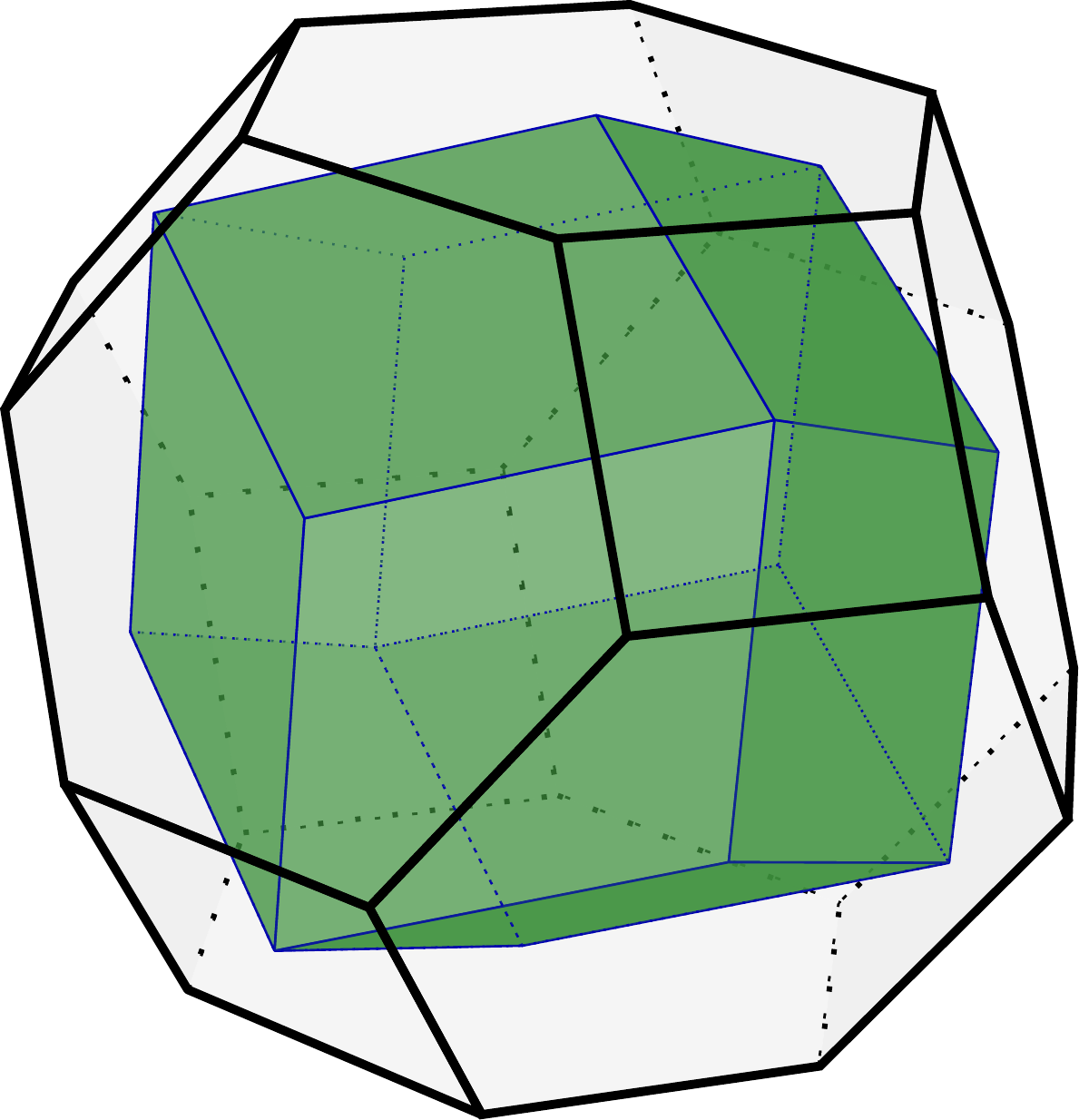}
\caption{For $n=3$: the permutahedron containing the zonotope}
\label{fig:permzono}
\end{figure}
\begin{proposition}\label{prop:zonotope}
The zonotope $Z=  \sum_{i=0}^n [0, \pi]e_{i}$ lies entirely in $\Perm$. Moreover, the vertices of $Z$ are the only points of $Z$ on the boundary of $\Perm$. (See Figure~\ref{fig:permzono} for the n=3 case.) 
\end{proposition}
\begin{proof}
The vertices of the zonotope $Z$ are labeled by proper 2-partitions $\{I_-,I_+\}$ of $\hat n$. 
%\ChHelge{with $I_-\neq\emptyset$, $I_+\neq\emptyset$}.
 They are in one-to-one correspondence with the facets of $\Perm$. It is enough to show that each vertex sits in the relative interior of the corresponding facet of $\Perm$. A vertex $v_{I_-,I_+}=[x_0, \dots ,x_n]$ has homogeneous coordinates
\begin{equation}
x_i= \
\begin{cases} 
\pi, i \in I_+\\ 
0, i \in I_-
\end{cases}
\end{equation}
Then we see that $v_{I_-,I_+}$ satisfies all inequalities \eqref{eq:perm_facets}, of which exactly one is an equality and the others are strict.
\end{proof}

\subsection{The coamoeba and its skeleton}\label{sec:coamoeba}
%Here we combine the well-known facts about the coamoeba side of the pair-of-pants. 
The image $\cc:=\mu_2(P)\subset \T^n$ is called the {\bf coamoeba} of $P$. 
First, some selected history of the subject: 
Mikhalkin \cite{PP} used Viro's patchworking \cite{Viro83,Vi18} to construct (torus) fibrations of hypersurfaces via projection to the spine of the amoeba with coamoeba fibers (see also \cite{Viro11}, \cite{NS} and \cite{KZ}).
%The fibers of the fibration jump in dimension unfortunately (and moreover aren't actually entirely known). 
The intention to have fibers equi-dimensional with full understanding of their geometry is one of the motivations for the present article (see Proposition~\ref{main-fact}). 
%In principle our treatment of the pair of pants case (combined with Mikhalkin's reasoning) easily allows to do general hypersurfaces with smooth tropicalization.
The first account of the skeleton of the coamoeba as the permutahedron to our knowledge is Futaki-Ueda \cite{FutakiUeda}.
% who also noticed this description generalizes dimers. 
On the symplectic side Sheridan \cite{Sheridan} first made use of the zonotope, the complement of the coamoeba, by viewing its boundary as an immersed Lagrangian sphere in the pair-of-pants.
More recently, Nadler-Gammage \cite{GN20} were able to view the skeleton of the coamoeba as a Lagrangian. A different Lagrangian skeleton had previously been given by \cite{RSTZ,PZ}.

When restricted to $P_0$, the corresponding part $\cc_0\subset \cc$ of the coamoeba is again the hypersimplex $\Delta(2) \subset \Delta_0$, since the only restrictions on the exterior angles $\alpha_i$ in a polygon are $\sum_i\alpha_i = 2\pi$ and $\alpha_i\leq \pi$. In particular, the vertices of $\cc$ are labeled by {\em unordered} 2-partitions  $\sigma=\<I_-, I_+\>$.

%\subsection{The skeleton $\ss$}
We will identify the torus $\TT^n$ with $(\R^{n+1}/\R)/2\pi\Lambda^*$ where $\Lambda^*$ is the coweight lattice from the previous section.
 %We now turn our attention to the skeleton  $\TT^n = (\R^{n+1}/\R)/2\pi\Lambda^*$. Let $\theta_i, \ i=0, \dots, n$ be the homogeneous coordinates.
% Let $\theta_i, \ i=0, \dots, n$ be the homogeneous coordinates in $\TT^n$ which we sometimes call {\bf angles}.
Let $\ss$ be the image of the {\em boundary} of $\Perm$ in $\TT^n$. Of the $(n+1)!$ vertices of $\Perm$ only $n!$ are distinct in $\TT^n$. They correspond to total {\em cyclic} orderings $\sigma=\<i_0, \dots, i_n\>$. The pairs of opposite facets of $\Perm$ are identified, thus giving $2^n-1$ facets of $\ss$ labeled by (unordered) 2-partitions $\sigma=\<I_-, I_+\>$. Each contains the corresponding vertex $v_\sigma$ of $\cc$ as its barycenter. 

More generally, the faces  $\ss _\sigma$ of $\ss$ are labeled by the cyclically ordered partitions $\sigma=\<I_1, \dots, I_k\>$. Thus the face lattice of $\ss$ is dual to the (truncated) lattice of the subdivision $\{\TT^n_\sigma \suchthat |\sigma|\ge 2\}$ of $\TT^n$. We call the intersection point $v_\sigma:=\ss_\sigma \cap \TT^n_\sigma$ the {\bf barycenter} of the face $\ss_\sigma$ ($\ss_\sigma$ and $ \TT^n_\sigma$ are of complimentary dimensions). We let 
$$\gamma_s:= \sum_{j=1}^{s-1} |I_j|+ \frac12 |I_s|, \quad s=1, \dots, k.
$$
Then in the homogeneous coordinates $\theta_i$ the barycenter of $\ss_\sigma$ is
\begin{equation}\label{eq:barycenter}
v_\sigma=\frac{2\pi}{n+1} [ \underbrace{\gamma_1,\dots, \gamma_1}_{I_1},\underbrace{\gamma_2,\dots,\gamma_2} _{I_2}, \dots, \underbrace{\gamma_k,\dots , \gamma_k}_{I_k} ]. 
\end{equation}
(see the interpretation in terms of the distinguished polygons $\mathcal D_\sigma$ below).
For $\sigma=\langle i_0,...,i_n\rangle$ a maximal cyclic partition, the vertex $v_\sigma$ of $\ss$ is the projection of each of the $n+1$ permutahedron vertices $\rho^\vee_{\bar\sigma}$ where $\bar\sigma$ is one of the $n+1$ many decyclizations of $\sigma$, cf.~\eqref{eq:rho}.

\begin{proposition}\label{prop:skeleton}
The coamoeba $\cc$ deformation retracts to its skeleton $\ss$.
\end{proposition}

\begin{proof}
This follows from the fact that the interior of the zonotope $Z$ is the complement of the coamoeba and $Z$ sits inside $\Perm$ by Proposition \ref{prop:zonotope}. 
\end{proof}

%As a corollary we have the following well known fact
%\begin{corollary}
%There is a lift  $\bar\ss$ of the skeleton to the pairs of pants $P$ and $P$ deformation retracts to  $\bar \ss$.
%\end{corollary}

For a more profound understanding of $\ss$ it is useful to write down the precise inequalities that cut out the faces of $\ss$. 
We start with facets. The hyperplane equality \eqref{eq:perm_facets} for a facet $\ss_{I_-,I_+}$ reads
\begin{equation}\label{eq:s-facets}
\frac{(\theta_{i_r}+\dots +\theta_{i_n})}{s} - \frac{(\theta_{i_0}+\dots +\theta_{i_{r-1}})}{r} = \pi.
\end{equation}
We can interpret this as follows. Given a subset $I\subset\bar n$, let 
$$\theta_{I}:= \frac1{|I|} \sum_{i\in I} \theta_{i}
$$ 
be the {\bf average argument} in $I$. Then \eqref{eq:s-facets} says that the average arguments $\theta_{I_-}$ and $\theta_{I_+}$ differ by $\pi$.
The inequalities that cut $\ss_{I_-,I_+}$ are the following. For any subset $I'_- \subseteq I_-$ the average argument $\theta_{I'_-}$ differs from $\theta_{I_-}$ by no more than $\frac\pi{n+1}(|I_-| -|I'_-|)$, and same for $I_+$.

Now we describe the defining conditions for $\ss_\sigma$ for an arbitrary partition  $\sigma= \<I_1, \dots, I_k\>$. Combining the  $2^{k-1}-1$ equations \eqref{eq:s-facets} associated to all possible 2-partition coarsenings of $\sigma$ we arrive at the following set of linear equations: 
\begin{equation}\label{eq:s-faces}
 \theta_{I_{s+1}} - \theta_{I_s} = \frac\pi{n+1}(|I_s|+|I_{s+1}|), \text{ for  all } s=1,\dots, k.
\end{equation}
To visualize this condition geometrically it is convenient to introduce a {distinguished} $k$-gon~$\mathcal D_\sigma$. 
Mark $n+1$ equally spaced points on a circle and let $\mathcal D_\sigma$ be the polygon whose set of vertices is the subset of points that separate the parts in the partition $\sigma = \<I_1, \dots, I_k\>$.
Then the equations \eqref{eq:s-faces} say that for a point in $\ss_\sigma$ the average arguments $\theta_{I_s}$ are equal to the arguments of the corresponding sides
%\Helge{Do you mean "..faces of.."?}
 of $\mathcal D_\sigma$.

The inequalities that bound the face $\ss_\sigma$ are similar to those for a facet: for any subset $I'_s \subset I_s$ the average argument in $I'_s$ differs from $\theta_{I_s}$ by no more than 
$\frac\pi{n+1}(|I_s| -| I'_s|)$. The vertices of $\ss$, which correspond to full partitions $\sigma=\<i_0, \dots, i_n\>$, have arguments of the regular $(n+1)$-gon.

\begin{remark}
There is a striking similarity between the amoeba and the coamoeba worlds. Indeed, when restricted to the simplex $\Delta_0$ the corresponding part $\cc_0\subset \Delta_0$ of the coamoeba looks exactly like the amoeba $\am \subset \Delta$, both are identified with the hypersimplex $\Delta(2)$. But here is a bit of warning. The corresponding part of the skeleton $\ss_0$ is not the same as the tropical skeleton $S$. The latter is truly symmetric with respect to all permutations of $\hat n$. On the other hand, $\ss_0$ retains only the dihedral symmetry of the $(n+1)$-gon. What we called the barycenters \eqref{eq:barycenter} of the faces $\ss_\sigma, \sigma\preceq \sigma_0$, are not generally the true barycenters of the corresponding simplicial faces of~$\Delta_0$. Since $S_3\cong D_6$ we don't see this distinction in the $n=2$ case.
\end{remark}

\section{The ober-tropical pair of pants and the isotopy}
\subsection{The ober-tropical pair-of-pants and its stratification}
Recall that the faces $S_{IJ}$ of $S$ are labeled by pairs $I\subseteq J\subseteq \hat n$ and the faces $\ss_\sigma$ of $\ss$ are labeled by cyclic partitions $\sigma=\< I_1,\dots,I_k\>$ of $\hat n$. For the dimensions, we have
\begin{equation} \label{eq-dim-formula}
\dim S_{IJ}=|J\setminus I|,\qquad \dim \Sigma_\sigma=n+1-|\sigma|.
\end{equation}
Now we define a new object, the {\bf ober-tropical} pair-of-pants $\pp$, as a subcomplex of $S\times\ss$:
 $$\pp := \bigcup_{\sigma,I, J} S_{IJ} \times \ss_\sigma,
 $$
where the union is over all triples $(\sigma, I \subseteq J)$ such that $\sigma$ divides $I$. 
For the $n=2$ case, the ober-tropical pair-of-pants are shown on the right in Fig. \ref{fig:pants2}. 

Similar to the cell structure $\{P_{\sigma, J}\}$ of $P$ the stratification $\{\Delta_J\times \TT^n_\sigma\}$ of $\Delta\times \TT^n$ induces a stratification $\{\pp_{\sigma, J}\}$ of $\pp$. As before we will drop the subscript $J$ if $J=\hat n$ and denote by $\pp_0$ the stratum corresponding to $\sigma_0=\{0,1,\dots,n\}$. Again, since $\{\pp_{\sigma, J}\}$ does not touch the vertices of $\Delta_\sigma$ (see \eqref{eq:lift}) we will view each stratum $\pp_{\sigma, J}$ as sitting inside the product of two simplices $\Delta_J \times \Delta_\sigma$.
The polyhedral faces of $\pp_{\sigma, J}$ are labeled by the pairs $(\tau\preceq \sigma', I\subseteq J')$ such that $\sigma'\preceq \sigma, J'\subseteq J$ and $\tau$ divides $I$. In particular, $\pp_{\sigma, J}$ is a subcomplex of $\dsd \Delta_J \times \dsd \Delta_\sigma$. The dimension of the face indexed by 
$(\tau\preceq \sigma', I\subseteq J')$ is
\begin{equation} \label{eq-dim-formula-2}
|J'|-|I|\,+\,|\sigma'|-|\tau|.
\end{equation}

In the $n=2$ case, we can visualize $\pp_0\subset \pp \subset \Delta\times \Delta_0$ as a hexagon in Fig.~\ref{fig:ober2}.
\begin{figure}[ht]
\centering
\includegraphics[width=.5\textwidth]{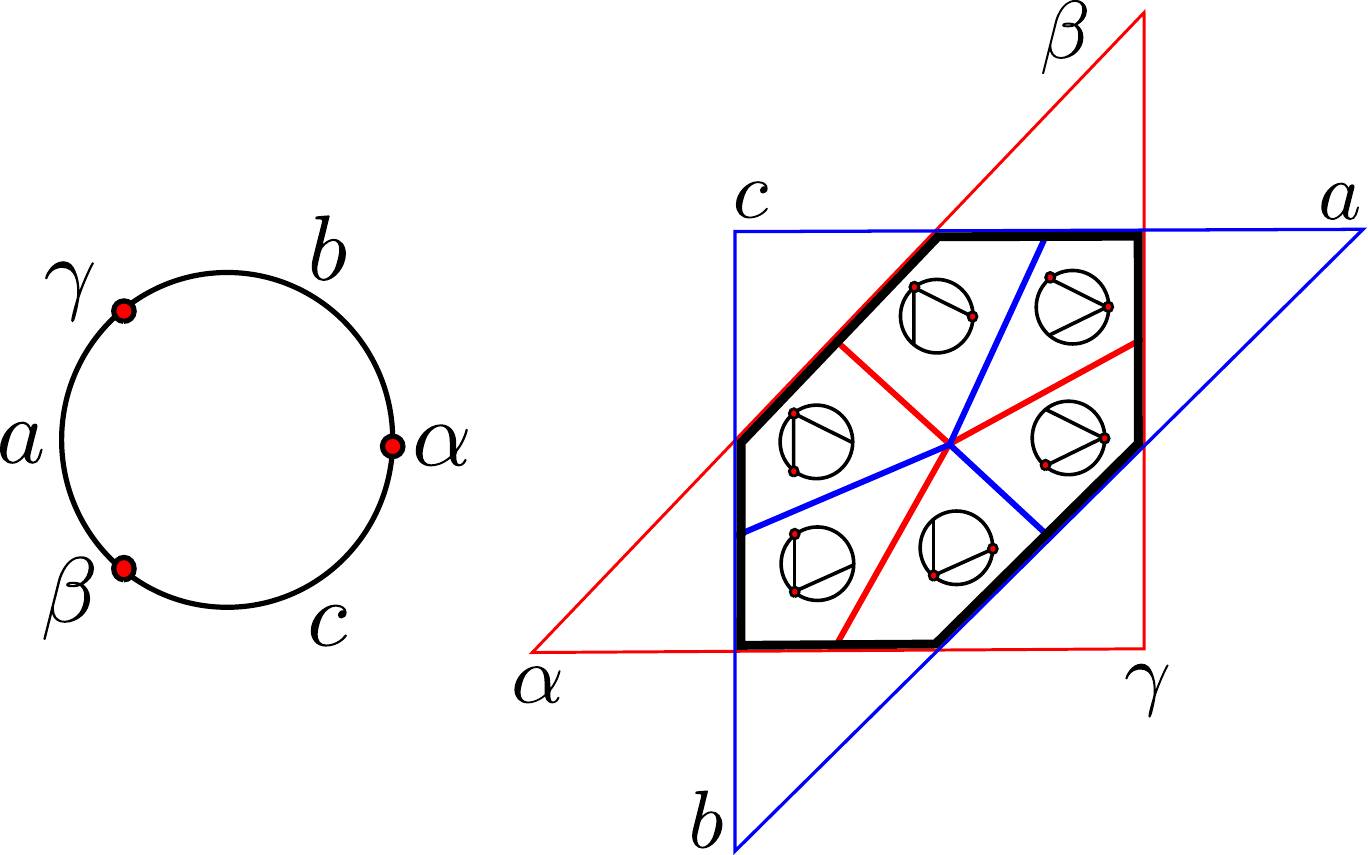}
\caption{The $n=2$ case: realization of $\pp_0$ as a hexagon made up of 6 squares.}
\label{fig:ober2}
\end{figure}
The entire $\pp$ is made up by gluing two of such hexagons along their three (red) sides, as is the classical complex pair-of-pants.
%Unfortunately this simple visualization does not work for $n\ge3$.

The ober-tropical pair-of-pants $\pp$ projects to both skeleta $\mu_1\colon\pp \to S$ and $\mu_2\colon \pp \to \Sigma$. We would like to describe the generic fibers of both maps. Given a maximal face $S_{ij}$ of $S$ let $T_{ij}$ be the $(n-1)$-torus in $\T^n$ defined by $\theta_i=\theta_j$. On the other side, for a maximal face $\Sigma_\sigma$ of $\Sigma$ (that is $\sigma=\<I_-,I_+\>$ is a 2-partition), let $\mathcal R_\sigma$ be the union of (the closures of) all facets $S_{ij}$ with $i\in I_-, j\in I_+$. 

\begin{proposition} 
\label{main-fact}
The two projections of $\pp$ to $S$ and to $\Sigma$ have $(n-1)$-dimensional fibers. The fiber $\mu_1^{-1}(y)$ over a generic point $y\in S_{ij}$ is homotopically equivalent to the torus $T_{ij}$. 
The fiber $\mu_2^{-1}(s)$ over a generic point $s\in \Sigma_\sigma$ is $\mathcal R_\sigma$ which is an $(n-1)$-ball.
% (a "real component" of $P$). 
\end{proposition}

\begin{proof}
One easily concludes from \eqref{eq-dim-formula} that the dimension of the fibers in both projections is $n-1$.

The fiber over a point $s\in \Sigma_\sigma$ is $\mathcal R_\sigma$ by definition. 
One can explicitly see that $\mathcal R_\sigma$ is isomorphic to $\dsd (\Delta_-\times \Delta_+)$ as a polyhedral complex.
Another way to see that $\mathcal R_\sigma$ is a ball is by invoking Viro's isotopy that identifies it with a component of the real locus, see 
Remark~\ref{viro} below.

For $\mu_1^{-1}(y)$ we need to combine the pieces of $\mu_1^{-1}(y)$ from different $\T^n_\sigma$'s. But each $\mu_1^{-1}(y) \cap \T^n_\sigma$ is again, similar to the Viro's patchworking, is isotopic to the subspace in the simplex $\Delta_\sigma$ defined by  $\theta_i=\theta_j$. Combined together they form the torus  $T_{ij}$.
\end{proof}

\begin{remark} \label{viro}
The compactified real locus in the complex pair of pants $P$ is the inverse image under the projection $P\rightarrow \cc$ of the set of points whose coordinates take values in $\{0,\pi\}$, or in other words, the inverse image of the set of barycenters $v_\sigma$ where $\sigma$ runs over the two-partitions, see \eqref{eq:barycenter}. We denote the inverse image of $v_\sigma$ by $R_\sigma$, it is the real component of $P$ defined by the choice of coordinate signs according to $\sigma=\<I_-,I_+\>$.
Analogously in the ober-tropical case, by definition, $\mathcal R_\sigma$ is naturally identified with the inverse image of $v_\sigma$ under the projection $\mu_2:\pp\rightarrow \Sigma$. An ambient isotopy inside $(\mathbb{R}^*)^n$ that takes $R_\sigma$ to $\mathcal R_\sigma$ was already given in
Viro's patchworking \cite{Viro83}, see also \cite[Theorem 5.6]{GKZ}.
\end{remark}

\begin{remark}
The fibers over generic points in $\Sigma$ being balls allows one to view $\pp$ as the total space of the  ``cotangent bundle'' $T^*\Sigma$ (a symplectic geometer's dream). Of course, some fibers are singular. For instance, the fiber over any vertex of $\Sigma$ is the entire $S$. Still all fibers are contractible which shows that $\pp$ contracts to $\Sigma$. This is as good as it gets if one wants to follow the philosophy that a symplectic manifold is the ``cotangent bundle'' over its Lagrangian skeleton.
\end{remark}

We will show that $\pp$ is, in fact, a topological manifold and the stratification $\{\pp_{\sigma, J}\}$ gives it a regular CW-structure. Then, since the two isomorphic regular CW-complexes are homeomorphic, the two versions of pairs-of-pants, complex and ober-tropical, are homeomorphic. The (much stronger) relative version of this homeomorphism is the main result of our paper.

% Here is the main result of the paper:
\begin{theorem}\label{thm:main}
The two spaces $P$ and $\pp$ are (ambient) isotopic in $\Delta\times \TT^n$. An isotopy can be chosen such that it respects the stratification $\{\Delta_J\times \TT^n_\sigma\}$. 
In particular, restricting the isotopy to the real locus $\Delta\times \{0,\pi\}^n$, it reproduces Viro's isotopy that identifies $R_\sigma$ with $\mathcal R_\sigma$ as in Remark \ref{viro}.
\end{theorem}

The key ingredient for the proof of Theorem \ref{thm:main} is a compatible collection of homeomorphisms between the cell pairs $(\Delta_J\times\Delta_\sigma, P_{\sigma, J})$ and $(\Delta_J \times\Delta_\sigma, \pp_{\sigma, J})$. Since any stratum is a face of a maximal stratum, which are all identical, we can restrict to the case $\pp_0\subset \Delta\times\Delta_0$. 
Using the graphical code, we can identify any subpartition $\sigma\preceq\sigma_0$ with a subset $V$ of marked vertices of the $(n+1)$-gon and think of faces $\Delta_J$ of $\Delta$ as subsets $J$ of marked edges (see Fig. \ref{fig:interlacing}). 
Now the faces of $\pp_0$ are labeled by two pairs $(I\subseteq J, V\subseteq W)$, where $I$ and $V$ are interlacing.
%, that is, not all arcs in $I$ are opposite to all points in $V$. 

\subsection{Unknotted ball pairs} Here we collect some basic results from PL topology which we will need to prove the isotopy. We will be concerned with ball pairs of codimension 2. A ball pair $(B^q, B^{q-2})$ is {\bf proper} if $\partial B^{q-2} =B^{q-2} \cap \partial B^q$.
The {\bf standard ball pair} $B^{q, q-2}$ is the pair of cubes $([-1,1]^q, \{0,0\}\times[-1,1]^{q-2})$. Its boundary is the {\bf standard sphere pair} $S^{q-1,q-3}$. We say that a ball or a sphere pair is {\bf unknotted} it is homeomorphic to the appropriate standard pair.

\begin{proposition}\label{prop:unknotted} In the PL category
\begin{enumerate}
\item
A locally flat sphere pair $(S^{q}, S^{q-2}), q\neq 4$, is unknotted if and only if $S^{q}\setminus S^{q-2}$ has the homotopy type of a circle.

\item
A proper locally flat proper ball pair $(B^q, B^{q-2}), q\neq 4$,
whose boundary is an unknotted sphere pair is itself unknotted if and only if
%which \red{is a face} in an unknotted sphere pair $(\partial B^{q+1}, \partial B^{q-1})$ is unknotted and 
$B^q\setminus B^{q-2}$ has the homotopy type of a circle.

\item A ball pair $(B^{q+1}, B^{q-1})$ which is a cone over an unknotted sphere pair $(S^{q}, S^{q-2})$ is unknotted.

\item 
A homeomorphism between the boundaries of unknotted balls extends to their interior. Moreover one can choose the extension to agree with any given extension on the subball.
\end{enumerate}
\end{proposition}

The first statement is true in TOP for $q\ge 4$, see Freedman \cite[Chapter 11]{Freedman} but we could not find a reference for the analog of the second statement in TOP. So to stay on the safe side we will provide a proof for our $(B^4,B^2)$ case ``by hand'' (see Proposition \ref{lemma:42}) and remain in the PL category. Similarly, we could not use (1) and (3) for the ober-tropical $(B^5,B^3)$ case and therefore, since an induction argument is involved, neither in any higher dimension. We could run the argument in TOP, but again we decided to do the $(B^5,B^3)$ case by hand (see Lemma \ref{lemma:53}) and stay in PL. 
Also we could not find an analogous statement to (1) for the ball pair (which should be true) which might have allowed simpler arguments for some of the statements below.

\begin{proof}[Proof of Proposition~\ref{prop:unknotted}] 
Part (1) is proved by Levine \cite{Levine} for $q\ge 6$, and via a surgery argument by Rourke \cite{Rourke} for $q=5$, cf.~\cite[Theorem~7.6]{RS}. The case $q=3$ was done by Papakiriakopoulos \cite{Papa}. Part (4) is \cite[Proposition~4.4]{RS}. Part (3) is clear. 
We give an argument for the non-trivial direction of (2), cf.~\cite[Proposition 7.5]{RS}: assume that $\partial B^q\setminus B^{q-2}$ is homotopy equivalent to a circle. We may glue to its boundary a standard ball pair to form a sphere pair $(S^q, S^{q-2})$. The resulting sphere pair is proper and locally flat, and $S^q\setminus S^{q-2}$ is homotopy equivalent to a circle, hence it is unknotted by part (1). 
The original ball pair is obtained from an unknotted sphere pair by removing a standard ball pair and so the result follows from 
\cite[Corollary 4.9]{RS}.
\end{proof}
 
%\begin{proof}
%The first part of the proposition is due to The third statement is e.g., \cite[Proposition 4.4]{RS} which works in TOP category as well.
%
%For the second part we will adopt the proof of Theorems 7.4 and 7.6 from \cite{RS}. Let $N$ be a regular neighborhood of $B^{q-2}$ in $B^q$, then $(N, B^{q-2})$ is an unknotted pair. Define $W$ to be the closure of $B^q\setminus N$ and $M_0:=W\cap N$. Choose a collar $C$ on $M_0$ in the closure of $\partial W\setminus M_0$ and define $M_1$ to be the closure of $\partial W \setminus (M_0 \cup C)$. Then $W\sim S^1$ by hypothesis, and $M_0, M_1 \sim S^1$ because $(\partial N,\partial B^{q-2})$ and $(\partial B^q,\partial B^{q-2})$ are unknotted. Hence, $W$ is an $h$-cobordism, hence an $s$-cobordism since the Whitehead group $\operatorname{Wh} (\Z)=0$, hence (by the relative s-cobordism theorem) a product. Therefore, $(B^q, B^{q-2})$ is unknotted. The other direction is obvious.
%\end{proof}
%
%The statement of the proposition holds for $q=5$ in the topological category due to Freedman, see e.g. \cite{Freedman}, but it fails in the PL category.
%
%\begin{remark} One could possibly replace the second part of the proposition with Stallings's string $(\R^q, Y)$ unraveled at infinity \cite{Stallings}, but we prefer to work with closed balls.
%\end{remark}

\subsection{Homeomorphism of ball pairs and the proof of Theorem \ref{thm:main}}
The main building block for the proof of Theorem~\ref{thm:main} is a homeomorphism of the pairs $(\Delta_J\times\Delta_\sigma , P_{\sigma, J})$ and $(\Delta_J\times\Delta_\sigma , \pp_{\sigma, J})$.

\begin{lemma}\label{lemma:proper_balls}
Both $(\Delta_J\times\Delta_\sigma , P_{\sigma, J})$ and $(\Delta_J\times\Delta_\sigma , \pp_{\sigma, J})$ are proper ball pairs.
\end{lemma}
\begin{proof}
For the complex pair-of-pants the result was proved in \cite[Prop. 5]{KZ}. In the ober-tropical case, we can identify the polyhedral complex $\pp_{\sigma, J}$ with the dualizing subdivision of $C^\vee_{\sigma, J}$, the dual polytope of the cyclic polytope, see remark after Proposition~\ref{prop:stanley}.
\end{proof}

\begin{lemma}\label{lemma:complement}
Both complements $\Delta_J\times\Delta_\sigma \setminus P_{\sigma, J}$ and $\Delta_J\times\Delta_\sigma \setminus \pp_{\sigma, J}$ are homotopically equivalent to a circle.
\end{lemma}

\begin{proof}
Let $W$ denote the set of vertices of $\sigma$. Let $L_{\sigma,J}$ be the subcomplex of $\Delta_J \times \Delta_\sigma$ consisting of pairs $( \sigma', J')$ such that $\sigma'$ does {\em not} divide $J'$. That is  $L_{\sigma,J}$ consists of {\em non-interlacing} pairs $(I,V)$, where $I\subseteq J$ and $V\subseteq W$. 
We claim that both complements in the assertion deformation retract to $|L_{\sigma,J}|$, which then collapses to a circle.

%\underline{
Step 1: $\Delta_J\times\Delta_\sigma \setminus \pp_{\sigma, J}$ retracts to $|L_{\sigma,J}|$.
%}
Recall that  $\pp_{\sigma, J}$ is a subcomplex of $\dsd\Delta_J \times \dsd\Delta_\sigma$ which consists of all pairs $(\tau\preceq \sigma', I\subseteq J')$ such that $\sigma'\preceq \sigma, J'\subseteq J$ and $\tau$ divides~$I$. 
The open star $\Star^\circ(\pp_{\sigma, J})$ is defined to be the union of faces in $\dsd\Delta_J \times \dsd\Delta_\sigma$ whose closure meets $\pp_{\sigma, J}$, that is, the open star $\Star^\circ(\pp_{\sigma, J})$ consists of all pairs $(\tau\preceq \sigma', I\subseteq J')$ such that $\sigma'$ divides $J'$. 
Consequently, the complement of $\Star^\circ(\pp_{\sigma, J})$ in $\Delta_J \times \Delta_\sigma$ is the $\dsd \times \dsd$ refinement of $L_{\sigma,J}$. 

Now we need to show that the complement of $\Star^\circ(\pp_{\sigma, J})$  in $\dsd\Delta_J \times \dsd\Delta_\sigma$ is a deformation retract of $\Delta_J\times\Delta_\sigma \setminus \pp_{\sigma, J}$. 
Consider a face $F:=(\tau\preceq \sigma', I\subseteq J')$ in $\Star^\circ(\pp_{\sigma, J})$ that is not contained in $\pp_{\sigma, J}$ and let $\overline{F}$ denote its closure in $\dsd\Delta_J \times \dsd\Delta_\sigma$.
The intersection $M:=\overline{F}\cap \pp_{\sigma, J}$ is a proper subcomplex of $\overline{F}$. We claim it suffices to show that $M$ is collapsible because then its regular neighborhood $N(M)$ in $\partial F$ is a ball (see, e.g.~\cite[Theorem 3.26]{RS}) and so its complement in $\overline{F}$ retracts to $\partial F \setminus N(M)$. Hence, inductively on dimension, we can retract all faces in $\Delta_J\times\Delta_\sigma \setminus \pp_{\sigma, J}$ from $\Star^\circ(\pp_{\sigma, J})$.

To see that $M$ is collapsible, we note that its faces are pairs $(\tau'\preceq \sigma'', I'\subseteq J'')$ with $\tau\preceq\tau'\preceq\sigma''\preceq\sigma'$, $I\subseteq I'\subseteq J''\subseteq J'$ such that 
$\tau'$ divides $I'$. For a fixed $(\tau',I')$ they form a Boolean lattice. 
This lattice consists of only a single element if and only if $\tau'=\sigma'$ and $I'=J'$. In all other cases, the Boolean lattice decomposes into a disjoint union of face-facet pairs. Each such pair can be collapsed to the empty set if both cells in the pair are not contained in any other cell of $M$. (The collapsing of pairs works as indicated in Figure~\ref{fig:g-collapse}.) The latter condition can be ensured by induction on the minimality of $(\tau',I')$ (so by \eqref{eq-dim-formula-2}, the maximality of the cell dimension).
Hence, following this induction, we can collapse the entire complex $M$ to the single vertex $(\tau'=\sigma', I'=J')$.

%\underline{
Step 2: $\Delta_J\times\Delta_\sigma \setminus P_{\sigma, J}$ retracts to $|L_{\sigma,J}|$.
%}
One may view the complement of the $(n-1)$-dimensional pair-of-pants in $(\C^*)^n$ as the $n$-dimensional pair-of-pants. 
In a similar vein we represent the complement $N_{\sigma, J}$ of a regular neighborhood of $P_{\sigma, J}$ in $\Delta_J \times \Delta_\sigma$ as the union of $k$
%\Helge{this said $n+1$ instead of $k$ before} 
cells $P_{\bar \sigma, \hat J}$, where $\bar \sigma$ is a decyclization of $\sigma= \<I_1,\dots , I_k\>$, that is a choice of $I_1$, and $\hat J=J\cup\{g\}$ contains one additional {\bf ghost} element $g$. 
Geometrically, the ghost is $g=-z_0-\dots-z_n$ and its phase ``breaks the necklace'' $\sigma$. 
We permit the case $|g|=0$ but still require $g$ to have a phase in order to be representing points in $N_{\sigma, J}$. 
The full face lattice of the regular CW-complex $N_{\sigma, J}$ is given by $(\hat\sigma', J') \preceq (\langle g,\bar\sigma\rangle, \hat J)$ for which $\hat\sigma'$ divides $J'$. 
Here, $\langle g,\bar\sigma\rangle$ refers to the cyclic partition obtained from the decyclization $\bar \sigma$ by inserting $\{g\}$ as a new part at the point where $\sigma$ was broken and then viewing it as a cyclic partition.
The case $|g|>0$ is equivalent to $g\in J'$.

The elements in the face lattice of $N_{\sigma, J}$ decompose into three groups: a) those with $g\not \in J'$; b) those with $g \in J'$ and $\hat\sigma'$ divides $J'\setminus\{g\}$; c) those with $g \in J'$ and $\hat\sigma'$ does not divide $J'\setminus\{g\}$. From an element in b) we may remove $g$ from $J'$ to have an element in a) 
\\[0mm]
\noindent
\begin{minipage}{\linewidth}
%\centering
\begin{minipage}{0.5\textwidth} 
and this yields a bijection of the sets a) and b). Furthermore the cell with $g$ removed is a facet of the original one and lies in the boundary of $N_{\sigma, J}$. This fact may be used to collapse the cells in a) and b) altogether (inductively starting with the largest-dimensional ones), so that $N_{\sigma, J}$ collapses to its subcomplex consisting of the cells in c), see Figure~\ref{fig:g-collapse}.

In the subcomplex of the cells in c), we combine together all those cells $(\hat\sigma', J')$ into a single cell $(\sigma',J')$ that have the property that
\end{minipage}
\hspace{0.02\textwidth}
\begin{minipage}{0.48\textwidth}
\begin{figure}[H]
%\centering
\includegraphics[width=.92\linewidth]{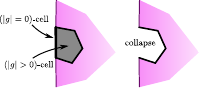}
\captionsetup{width=\linewidth}
\caption{Collapsing two cells that form a pair under the bijection of a) and b).}
\label{fig:g-collapse}
\end{figure}
\vspace{0mm} 
\end{minipage}
\end{minipage}
removing $g$ from $\hat\sigma'$ yields $\sigma'$. (The counterclockwise order on the ghost position provides a shelling.) Note that the resulting CW complex is isomorphic to $L_{\sigma,J}$.

%\underline{
Step 3: $|L_{\sigma,J}|$ has the homotopy type of a circle.
%:}
We first collapse $L_{\sigma,J}$ to a (generally) 2-dimensional subcomplex, the {\bf belt}, which then collapses to a circle (see Fig. \ref{fig:belt}). 
\begin{figure}[h]
\centering
\includegraphics[height=40mm]{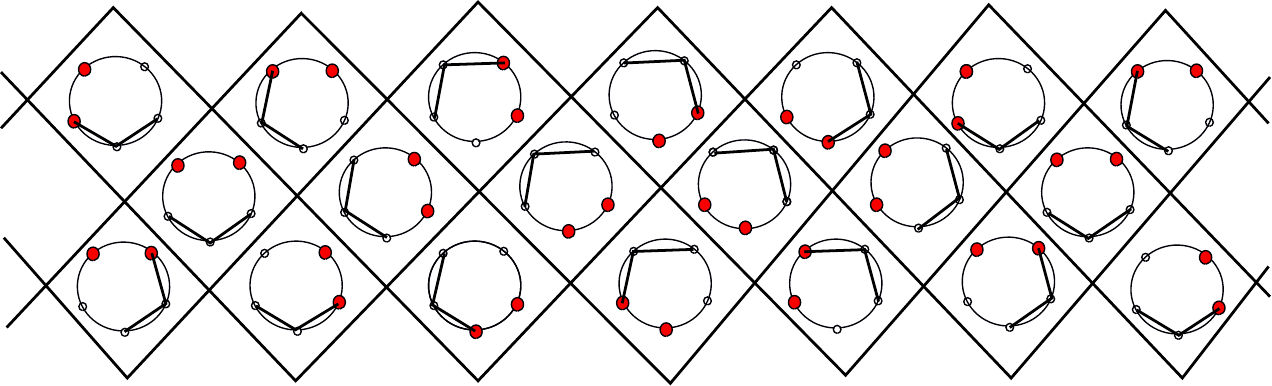}
\caption{The result of collapsing $L_{\sigma,J}$ to its belt for $n+1=5$.}
\label{fig:belt}
\end{figure}

Given a non-interlacing configuration $(I,V)\subseteq(J,W)$ we define four integers, the {\bf distances} between $I$ and $V$, as follows. In the counterclockwise order let $v_f\in V$ be the first vertex in the string of vertices $V$ and $v_l\in V$ be the last vertex. Similar let $e_f, e_l \in I$ be the first and the last edges in the string of edges. We let $k_e$ be the number of edges in $J$ and $k_v$ be the number of vertices in $W$ between $v_l$ and $e_f$. Similarly, let $m_e, m_v$ be the numbers of edges in $J$ and vertices in $W$ between $e_l$ and $v_f$. Note that $(I,V)$ is a face of $(I',V')$ only if their sets $(v_f, v_l, e_f, e_l)$ are the same, or $k'_e\ge k_e, k'_v \ge k_v ,m'_e\ge m_e, m'_v \ge m_v$ with at least one inequality strict. 

%We give a partial order on the set of non-interlacing configurations $\{(I,V)\}$ by saying $(I',V')\succeq (I,V)$ if $k'_e\ge k_e, k'_v \ge k_v ,m'_e\ge m_e, m'_v \ge m_v$ with the condition that if all are equalities then $(I',V')=(I,V)$.

We collapse the intervals $I_{(v_f, v_l, e_f, e_l)}$ of the configurations with fixed sets $(v_f, v_l, e_f, e_l)$ starting from those with $(k_e, k_v, m_e, m_v)=(0,0,0,0)$ and increasing the distances. That is, we collapse $I_{(v_f, v_l, e_f, e_l)}$ after all intervals $I_{(v'_f, v'_l, e'_f, e'_l)}$ with $k'_e\ge k_e, k'_v \ge k_v ,m'_e\ge m_e, m'_v \ge m_v$ with at least one inequality strict had been collapsed. Any such interval is either Boolean (then it collapses), or consist of a single element (then it stays). Note that the elements in $I_{(v_f, v_l, e_f, e_l)}$ are faces of its own elements or faces of elements with lower distances (which we collapsed already by induction). 

Thus we end up with the subcomplex, which we call the {\bf belt} of $L_{\sigma,J}$, of non-interlacing pairs $(I,V)$ such that $I$ has one or two consecutive elements in $J$, and $V$ consists of one or two consecutive vertices in $W$. The belt is a 2-dimensional complex whose 2-faces are squares (see Fig.~\ref{fig:belt}) labeled by consecutive non-interlacing pairs ($v_f$ next to $v_l$) and  ($e_f$ next to  $e_l$), unless $|J|=2$ or $\sigma$ consists of two parts, in which case, either $v_f=v_l$ or $e_f=e_l$, and it is already a circle. The belt has two coordinate directions: in one direction the configuration of edges does not change, in the other direction the vertices are fixed.

Finally we collapse the belt to a circle. It's not as pretty as before since we will have to make a choice between $k$'s and $m$'s. We choose $k$'s and collapse the belt to an 1-dimensional subcomplex $S^1_{\sigma,J}$ (it will be the upper boundary in  Fig.~\ref{fig:belt}). The vertices of  $S^1_{\sigma,J}$ are pairs $(v,e)$ which are distance $(k_e, k_v)\leq(1,1)$ apart. The edges  of $S^1_{\sigma,J}$ are $(V,e)$ and $(v,E)$ where $V$ is a consecutive pair of vertices and $E$ is a consecutive pair of edges with $(k_e, k_v)=(0,0)$. So that $|S^1_{\sigma,J}|=S^1$.
We collapse the squares one by one starting with $(m_e, m_v)=(0,0)$ and increasing the distances $(m_e, m_v)$ as we did before.
\end{proof}

\begin{corollary}\label{cor:complement}
The complements of the sphere pairs $\partial(\Delta_J\times\Delta_\sigma) \setminus \partial P_{\sigma, J}$ and $\partial(\Delta_J\times\Delta_\sigma) \setminus \partial \pp_{\sigma, J}$ have the homotopy type of a circle.
\end{corollary}
\begin{proof}
The deformation retractions of
$\Delta_J\times\Delta_\sigma \setminus P_{\sigma, J}$ and $\Delta_J\times\Delta_\sigma \setminus \pp_{\sigma, J}$ to
$L_{\sigma,J}$ in the previous lemma are obtained by a sequence of cell collapses and this clearly respects the boundary.
Since the complex $L_{\sigma,J}$ (which eventually retracts to the circle) sits in the boundary of the ball $\Delta_J\times\Delta_\sigma$ in both cases, we proved the corollary. 
\end{proof}

\begin{lemma}\label{lemma:53}
Let $\dim (\Delta_J\times\Delta_\sigma) \le 5$. The ball pair $(\Delta_J\times\Delta_\sigma , \pp_{\sigma, J})$ is unknotted.
\end{lemma}
\begin{proof}
For $\dim (\Delta_J\times\Delta_\sigma) \le 3$ the statement is trivial. For $\dim (\Delta_J\times\Delta_\sigma) =4$ the 2-ball is the cone over the circle in $S^3$ which is unknotted since its complement is homotopically equivalent to $S^1$ by Corollary~\ref{cor:complement}. For $\dim (\Delta_J\times\Delta_\sigma) =5$ the cases $\Delta_J^4 \times \Delta_\sigma^1$ and $\Delta_J^1 \times \Delta_\sigma^4$ are easily reduced to just one factor as in the beginning of the proof of Proposition \ref{lemma:42}. Similar reduction happens for not maximally interlacing $\Delta^3_J\times\Delta^2_\sigma$ and $\Delta^2_J\times\Delta^3_\sigma$.
Thus the only non-trivial case to consider is $\Delta^3_J\times\Delta^2_\sigma$ where $J$ and $\sigma$ are maximally interlacing. The other case $\Delta^2_J\times\Delta^3_\sigma$ is similar.

Let $\sigma$ be a 3-partition $\<I_1, I_2, I_3\>$ and let $J=\{0,1,2,3\}$ split its elements, say, as $0,1\in I_1,\ 2\in I_2,\ 3\in I_3$. We list the maximal faces of $\pp_{\sigma, J}$ (we will omit the $J$ part from the subscript of the faces $S_{IJ}$ and denote the 3 coarsenings of $\sigma$ as $a=\<I_1 \cup I_2, I_3 \>,\ b=\<I_1, I_2 \cup I_3\>,\ c=\<I_3 \cup I_1, I_2\>$):
\begin{align*}
S_{23}\times \Sigma_a, \ S_{03}\times \Sigma_a, \ S_{13}\times \Sigma_a, &   \\
S_{03}\times \Sigma_b, \ S_{13}\times \Sigma_b,  & \ S_{02}\times \Sigma_b, \  S_{12}\times \Sigma_b, \\
& \ S_{02}\times \Sigma_c, \ S_{12}\times \Sigma_c,\  S_{32}\times \Sigma_c.
\end{align*}
\begin{figure}[h]
\centering
\includegraphics[height=45mm]{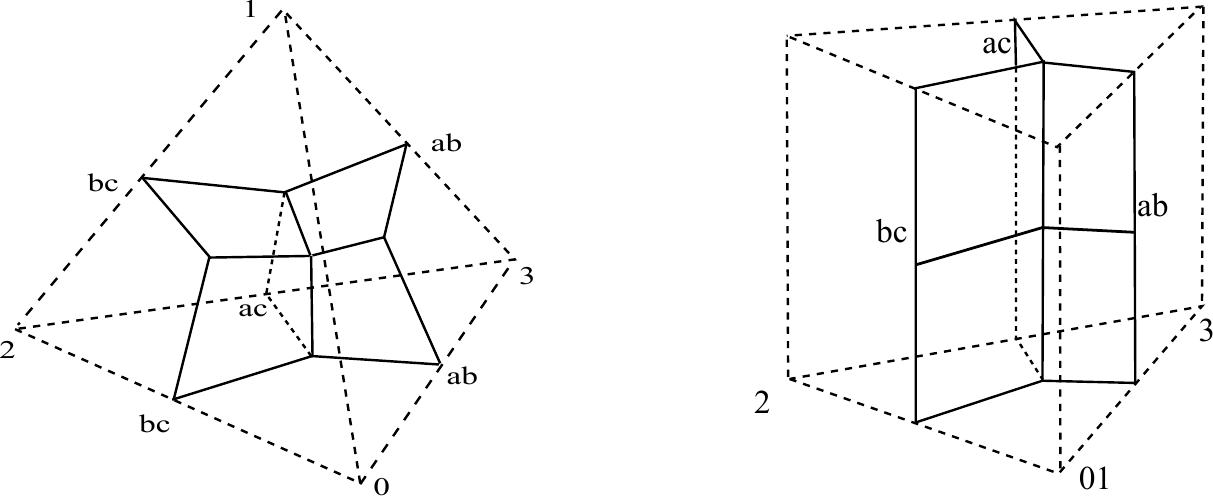}
\caption{The $S$-part of the complex $\pp_{\sigma, J}$ with $\Sigma_{xy}$-fibers labeled.}
\label{fig:53}
\end{figure}
Notice that the face $S_{01}$ is not present and one can combine the faces $S_{0x}$ and $S_{1x}$ into a single cell. Now it is easy to see that the entire complex $\pp_{\sigma, J}$ is a subdivision of the (properly subdivided at the boundary) complex $\pp_{\sigma, \bar J} \times [0,1]$ (see Figure \ref{fig:53}), where in $\bar J$ we treat 0 and 1 as a single element. Moreover this PL homeomorphism extends to the ambient balls, that is, the pair $(\Delta_J\times\Delta_\sigma , \pp_{\sigma, J})$ is homeomorphic to the pair $(\Delta_{\bar J}\times\Delta_\sigma , \pp_{\sigma, \bar J}) \times [0,1]$.
\end{proof}

\begin{proposition}\label{prop:ober_balls}
The ball pair $(\Delta_J\times\Delta_\sigma , \pp_{\sigma, J})$ is unknotted. 
\end{proposition}

\begin{proof}
We will do the induction on dimension of the strata $\Delta_J\times\Delta_\sigma$.  For $\dim(\Delta_J\times\Delta_\sigma)\le 5$ we have Lemma \ref{lemma:53}.
In general, $(\Delta_J\times\Delta_\sigma , \pp_{\sigma, J})$ is a cone pair over its boundary. So that its unknotness (and hence the local flatness) follows from unknotness of the boundary pair $(\partial (\Delta_J\times\Delta_\sigma), \partial \pp_{\sigma, J})$ by applying part (3) of Proposition~\ref{prop:unknotted}. To show that the boundary pair $(\partial (\Delta_J\times\Delta_\sigma ), \partial \pp_{\sigma, J})$ is unknotted we apply part (1) of Proposition \ref{prop:unknotted}. The complement is homotopically equivalent to the circle by Corollary~\ref{cor:complement}, the local flatness within the relative interiors of the strata of $\Delta_J\times\Delta_\sigma$ is by induction.

It only remains to show that $\pp_{\sigma, J}$ is locally flat at a stratum $\Delta_{J'}\times\Delta_{\sigma'} \subseteq \Delta_J\times\Delta_\sigma $. Note that each face $S_{IJ}$ of $S$ at the stratum $\Delta_{J'}$, $I\subseteq J'\subseteq J$, is locally the product  $(S_{IJ}\cap \Delta_{J'}) \times \Cone_{JJ'}$, where $\Cone_{JJ'}$ is the relative cone of $\Delta_J$ at its face $\Delta_{J'}\subseteq \Delta_J$. Similarly, any face of $\Sigma_\sigma$ is locally a product $(\Sigma_\sigma \cap \Delta_{\sigma'}) \times \Cone_{\sigma\sigma'}$. Thus the entire pair $(\Delta_J\times\Delta_\sigma , \pp_{\sigma, J})$ at $\Delta_{J'}\times\Delta_{\sigma'}$ is locally the product of the pair $(\Delta_{J'}\times\Delta_{\sigma'}, \pp_{\sigma', J'})$ with the cone $\Cone_{JJ'} \times \Cone_{\sigma\sigma'}$ and the local flatness follows by induction from a lower dimension.
\end{proof}

\begin{lemma}\label{lemma:42}
Let $\dim (\Delta_J\times\Delta_\sigma) \le 4$. The ball pair $(\Delta_J\times\Delta_\sigma , P_{\sigma, J})$ is unknotted. 
\end{lemma}

\begin{proof}
If $\dim (\Delta_J\times\Delta_\sigma) \le 2$, the statement is trivial. 
For $\dim (\Delta_J\times\Delta_\sigma) = 3$, necessarily $\Delta_J\times\Delta_\sigma$ is a prism $\Delta^2\times [0,1]$ and 
$P_{\sigma, J}$ is the a segment $S\times \{\frac12\}$ where $S\subset \Delta^2$ connects the midpoints of two sides of the triangle.

We prove the case $\dim (\Delta_J\times\Delta_\sigma) =4$.
In the case $\Delta_J^3 \times \Delta_\sigma^1$ there are 4 elements in $J$  and $\sigma$ is a 2-partition. The 2-ball $P_{\sigma, J}$ is the product $v\times R_\sigma$ where $v$ is the center of the interval $\Delta^1_\sigma$  and $R_\sigma$ is the real component of the 2-dimensional pair-of pants, where the elements in $J$ carry plus or  minus signs according to which of the two parts of $\sigma$ they belong. It is easy to trace the image of $R_\sigma$ in $\Delta^3_\sigma$ as a 2-disc separating the vertices according to $\sigma$. The case $\Delta_J^1 \times \Delta_\sigma^3$ is similar. If $\sigma$ and $J$ are both of size 3 and {\em not} maximally interlacing then the ball $P_{\sigma, J}$ in $\Delta_J^2 \times \Delta_\sigma^2$ is the product of two intervals each connecting midpoints of a pair of edges in its own triangle.

The only nontrivial case is the classical one-(complex)-dimensional pair-of-pants: the hexagon $P_0$ in the product of the 2 triangles $\Delta^2_0\times \Delta^2$. 
Similar to the proof of the complement being homotopically equivalent to $S^1$ in Lemma \ref{lemma:complement}, we base our argument on viewing the complement of the $1$-dimensional pair-of-pants in $(\C^*)^2$ as the $2$-dimensional pair-of-pants. 
More precisely we represent the complement $N_0$ of a regular neighborhood of $P_{0}$ in $\Delta^2 \times \Delta^2_0$ as the union of three cells $P_{\hat g}$, where the position of the ghost $g$ between the cyclically ordered $\{0,1,2\}$ is recorded, i.e.~$\hat g\in\{\{g,0,1,2\}, \{0,g,1,2\}, \{0,1,g,2\}\}$, see also Figure~\ref{fig:coamoeba2}. (As a reminder for clarity: recall that there is a second hexagon which arises from the other cyclic ordering $\{0,2,1\}$ and which also sits inside three $4$-balls giving the total of six strata as discussed in the $n=3$ case of Figure~\ref{fig:torus-strat}.)
\renewcommand\sidecaptionsep{1cm}
\begin{SCfigure}[][h]
%\begin{figure}[h]
\centering
\includegraphics[height=48mm]{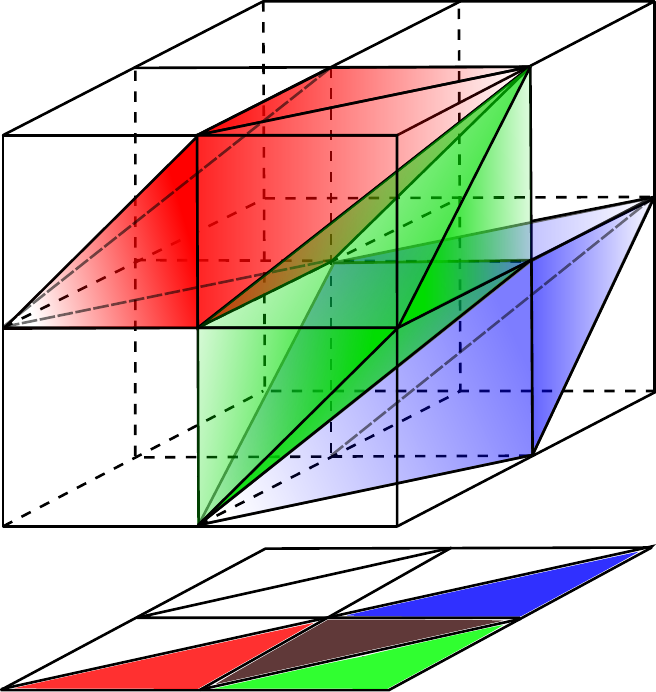}
\caption{Thinking of $\Delta^2 \times \Delta^2_0$ as the union of three 4-dimensional cells $P_{\hat g}$ which project to these octahedra in the argument factor of $(\C^*)^3$.}
\label{fig:coamoeba2}
\end{SCfigure}

The boundary of the regular neighborhood of $P_0$ decomposes naturally as $SP_0\cup ((\partial P_0)\times B^2)$ where $SP_0$ is a trivial circle bundle over $P_0$ and $B^2$ is a disc.
To show that the proper ball pair $P_0 \subset \Delta^2_0\times \Delta^2$ is unknotted, it suffices to show that there exists a homeomorphism
\\[0mm]
\noindent
\begin{minipage}{\linewidth}
%\centering
\begin{minipage}{0.6\textwidth} 
$N_0\stackrel\sim\to SP_0\times [0,1]$ where we like to require the homeomorphism to restrict to the identity on $SP_0\times\{ 0\}$ in the sense that we have $SP_0$ naturally be part of the boundary of $N_0$.

We define $W_{\hat g}:=SP_0\cap P_{\hat g}$ and will show $W_{\hat g}$ is a three-ball for each $\hat g$. 
Since $N_0$ is glued from the three $4$-balls $P_{\hat g}$, in order to find the required homeomorphism
$N_0\stackrel\sim\to SP_0\times [0,1]$
we may instead show for each $\hat g$ that $P_{\hat g}$ is homeomorphic to $W_{\hat g}\times[0,1]$ in such a way that their gluing respects the product decomposition.
We will call $W_{\hat g}$ a {\bf wedge}. The boundary of each wedge decomposes into two squares and one hexagon and the squares are where a wedge meets the other wedges, that is, the wedges have identical shape and glue as shown in Figure~\ref{fig:glue-wedges}.
\end{minipage}
\hspace{0.05\textwidth}
\begin{minipage}{0.35\textwidth}
\begin{figure}[H]
%\centering
\includegraphics[height=48mm]{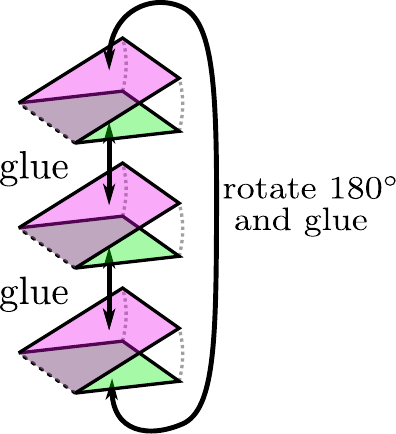}
\captionsetup{width=.95\linewidth}
\caption{The three $W_{\hat g}$ and how they glue to give $SP_0$.}
\label{fig:glue-wedges}
\end{figure}
\end{minipage}
\end{minipage}

We now construct the desired homeomorphism $P_{\hat g}\stackrel\sim\to W_{\hat g}\times[0,1]$. 
We start off by analyzing the structure of the polytopes $P_{\hat g}$. 
Recall that the faces of $P_{\hat g}$ are indexed by graphical code, that is by interlacing subsets of the circle with four distinct points on it and the arcs between the points labeled according to ${\hat g}$ in clockwise order.
In particular, there are $8$ facets in each $P_{\hat g}$, given by unmarking either a vertex or an arc while everything else is marked. 
Each $P_{\hat g}$ has a special facet given by unmarking the ghost and this facet is $W_{\hat g}$ because unmarking the ghost means setting $g=0$, i.e. $z_0+z_1+z_2=0$ however the phase of $g$ is still recorded having the effect of $W_{\hat g}$ being inside $SP_0$ rather than $P_0$ itself. We pictorially explain in Figure~\ref{wedge-graphical} that the intersection of two $P_{\hat g}$ is a facet of each.
\begin{figure}[h] 
\includegraphics[height=20mm]{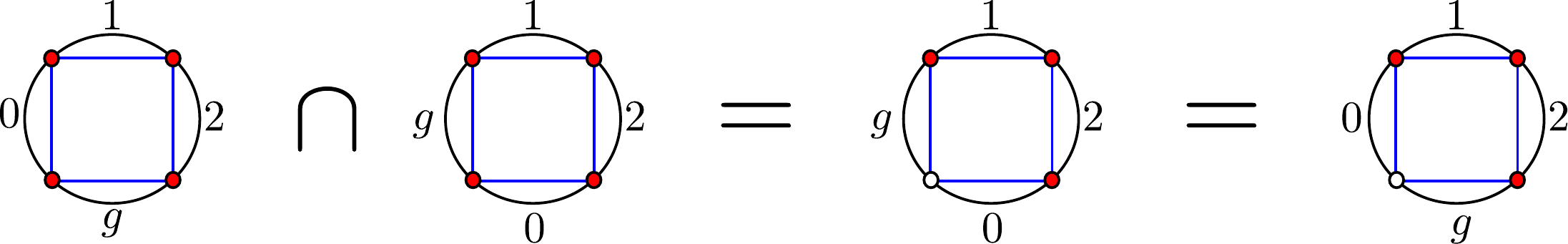}
\caption{Computing $P_{\{g,0,1,2\}}\cap P_{\{0,g,1,2\}}$ in graphical code.}
\label{wedge-graphical}
\end{figure}
Indeed, note that when a vertex is unmarked, the adjacent labels don't have a preferred order since they lie in the same $I_j$.
In summary, each $P_{\hat g}$ has three types of facets: 1) $W_{\hat g}$, 2) two facets that connect it with the two other $P_{\hat g}$ and 3) five further facets characterized by that the ghost and each of its adjacent vertices are marked.
Nonetheless, all eight facets of each $P_{\hat g}$ are isomorphic and look like the polytope in the center of Figure~\ref{fig:wedge}.

An important fact is that all three $W_{\hat g}$ meet in a disjoint union of three edges. We call them {\bf special edges} and their graphical code is given on the right in Figure~\ref{fig:wedge} (and the edges are bold). 
The figure shows $W_{\{g,0,1,2\}}$ in the center and explains on the right how $W_{\{g,0,1,2\}}$ gets identified with the wedge that we introduced before, i.e.~with a $3$-ball whose boundary decomposes in two squares and one hexagon. The original hexagonal face on top of $W_{\{g,0,1,2\}}$ becomes a square so that the special edges it contains lie opposite to one another - similarly for the bottom face. The remaining part of the boundary becomes a hexagon in the process. 
Convincing ourselves that the illustration on the left of Figure~\ref{fig:wedge} accurately depicts how $W_{\{g,0,1,2\}}$, $W_{\{g,0,1,2\}}$ and $W_{\{g,0,1,2\}}$ glue together, we have thus verified the structure of $SP_0$ as glued by wedges claimed in Figure~\ref{fig:glue-wedges}. 
\begin{figure}[h]
%\centering
\includegraphics[height=48mm]{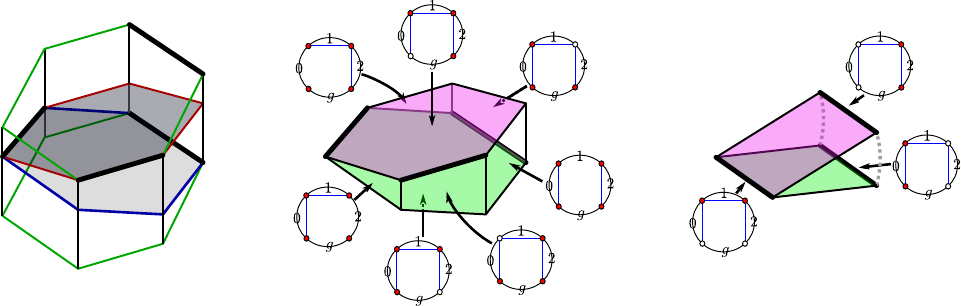}
\captionsetup{width=.95\linewidth}
\caption{The middle shows $W_{\{g,0,1,2\}}$ with each facet labeled by its graphical code and special edges marked bold. The right shows the homeomorphic version of it as a wedge with the special edges labeled by their graphical code. The left shows $W_{\{0,g,1,2\}}$ (top), $W_{\{g,0,1,2\}}$ (middle) and $W_{\{0,1,g,2\}}$ (bottom) and the top and bottom face of the resulting hexagonal prism shall be identified by translation to give $SP_0$, cf.~Figure~\ref{fig:glue-wedges}.}
\label{fig:wedge}
\end{figure}

In the following, we will use the terminology of ``square'' (respectively ``hexagon'', etc.) either to refer to actual squares (or hexagons, etc.) or to polyhedral complexes that glue to 2-cells whose boundary decomposes into 4 (or 6) intervals. For instance, the hexagonal top face in the middle of Figure~\ref{fig:wedge} is a square in the sense that its homeomorphic image in the wedge is a square. Similarly, the complement of the two actually hexagonal faces of the polytope in the middle is a ``hexagon'' in the sense that its homeomorphic image in the wedge is a hexagon.

To finish the proof, we are going to exhibit for each $P_{\hat g}$ a second copy $W'_{\hat g}$ of a wedge in its boundary disjoint from $W_{\hat g}$. 
We will show that each $P_{\hat g}$ is PL homeomorphic to $W_{\hat g}\times [0,1]$ where $W_{\hat g}\times\{0\}$ is the original $W_{\hat g}$ and $W_{\hat g}\times\{1\}$ is the extra copy; furthermore the gluing of the $P_{\hat g}$ will respect the product structure. 

We first observe that the ``opposite end'' of $N_0$, complementary to $SP_0$, is a Moebius strip. 
To be precise, the union of faces of all $P_{\hat g}$ that do not meet any $W_{\hat g}$ forms the strip in Figure~\ref{fig:Moebius} where the left and right non-bold boundary edges get identified according to the endpoint labels.
\begin{figure} %[H]
%\centering
\includegraphics[width=.9\linewidth]{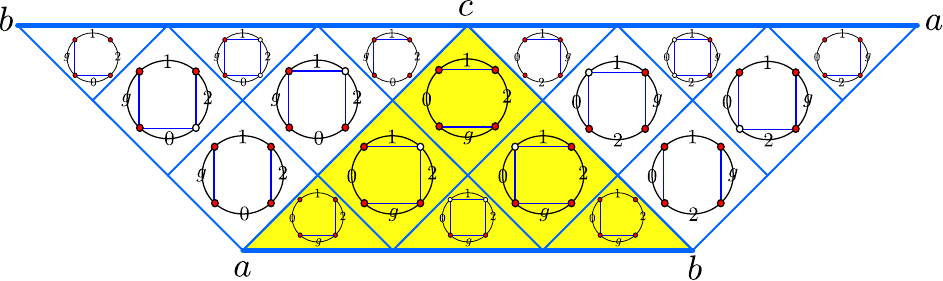}
\captionsetup{width=.8\linewidth}
\caption{The Moebius strip of cells in $N_0$ that are disjoint from $SP_0$.}
\label{fig:Moebius}
\end{figure}
The strip decomposes into three triangles, each of which lies in one $P_{\hat g}$. 
The center triangle in the figure (yellow) is contained in $P_{\{g,0,1,2\}}$, the left one in $P_{\{0,g,1,2\}}$ and the right one in $P_{\{0,1,g,2\}}$. 
Focusing on a single $P_{\hat g}$ now, we next need to understand better its facets that meet the corresponding triangle of the strip. 
Recall that the facets fall into three groups. 
The group characterized by having the ghost and its adjacent vertices marked consists of five facets. 
To understand how these glue together, consider the left of Figure~\ref{fig:blob-top} that shows these facets partially glued to a hexagonal tower. 
An additional gluing has to be carried out which affects the back side of the tower, a square shown in the middle of Figure~\ref{fig:blob-top}. The square decomposes into two triangles along the (bold) diagonal. The dashed triangle gets glued to the solid triangle resulting in a (blue) triangle which is precisely the triangle that is part of the Moebius strip. The result of the gluing is shown on the right in Figure~\ref{fig:blob-top}, let's call it the {\bf blob}.
\begin{figure} %[H]
%\centering
\includegraphics[width=.8\linewidth]{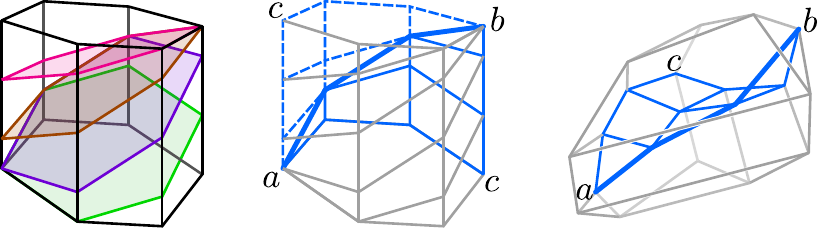}
%\captionsetup{width=.95\linewidth}
\caption{The five facets of $P_{\{g,0,1,2\}}$ that each meet the (blue) triangular part of the Moebius strip in a union of 2-cells. The situation looks similar for the other $P_{\hat g}$ after changing the labels $a,b,c$ appropriately.}
\label{fig:blob-top}
\end{figure}
%The remainder of the boundary of $P_{\hat g}$ consists of two of its facets. Each has a hexagon in common with $W_{\hat g}$ and meets the blue triangle in three edges: in the case of $P_{\{g,0,1,2\}}$ these are the edges connecting $a$ and $c$ for one facet and the edges connecting $b$ and $c$ for the other. 

We next place a wedge around the triangle inside the blob. Figure~\ref{fig:blob-top-nbd} depicts how this is going to work. 
The two squares of the wedge are the intersection of the wedge with the boundary of the blob and depicted on the left. The hexagonal face of the wedge is not shown to not overburden the illustration. 
It \emph{wraps around} the blue triangle, so that the resulting wedge - let us call it $W'_{\hat g}$ - contains the blue triangle in the way depicted in the middle of Figure~\ref{fig:blob-top-nbd}. 
\begin{figure} %[H]
%\centering
\includegraphics[width=.8\linewidth]{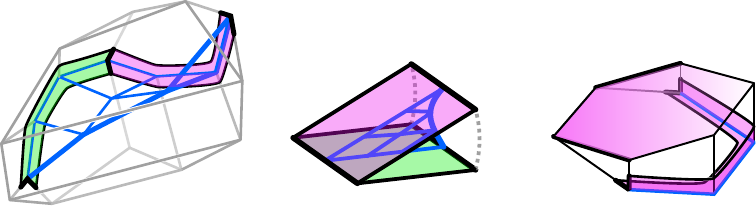}
%\captionsetup{width=.95\linewidth}
\caption{The wedge $W'_{\hat g}$ in the blob, containing the blue triangle (left and middle). The right shows one of three 3-cells that are the intersection of two $P_{\hat g}$'s.}
\label{fig:blob-top-nbd}
\end{figure}
Most importantly, we can decompose the blob into the wedge $W'_{\hat g}$ and $H\times[0,1]$ where $H$ is the hexagonal face of $W'_{\hat g}$ (the one that is not depicted). 
The next important observation is that the union of those faces of the blob (right of Figure~\ref{fig:blob-top}) which do \emph{not} meet the blue triangle is precisely the complement of the two hexagonal facets in the boundary of $W_{\hat g}$. In other words, the union of the blob and $W_{\hat g}$ (with their natural gluing) is a copy of two wedges ($W_{\hat g}$ and $W'_{\hat g}$) glued along $H\times [0,1]$ so that $W_{\hat g}$ meets $H\times [0,1]$ in $H\times \{0\}$ and $W'_{\hat g}$ meets $H\times [0,1]$ in $H\times \{1\}$. 
Next recall the facets of $P_{\hat g}$ that are also contained in the neighboring $4$-cells. 
We depict one of these on the right of Figure~\ref{fig:blob-top-nbd} with the purpose of illustrating how it contains two (pink) squares: one that it has in common with $W_{\hat g}$ (in hexagonal form) and the other being the matching square of $W'_{\hat g}$. 
We observe that we may view the facet as $S\times [0,1]$ where $S\times \{0\}$ is the pink (hexagonal) square of $W_{\hat g}$ and $S\times \{1\}$ is the pink square of $W'_{\hat g}$. 

Note that we have just given a PL homeomorphism of the boundary of $P_{\hat g}$ and the boundary of $W_{\hat g}\times [0,1]$. Any such homeomorphism can be extended to the interiors (e.g. by taking stars), so we have found a homeomorphism
$P_{\hat g} \stackrel\sim\rightarrow W_{\hat g}\times [0,1]$ for each $\hat g$. By construction, these homeomorphisms are compatible with gluing the three $P_{\hat g}$ so that we have constructed the desired PL homeomorphism 
$N_0\stackrel\sim\rightarrow SP_0\times [0,1]$.
\end{proof}

\begin{proposition}\label{prop:complex_balls}
%Let $\dim (\Delta_J\times\Delta_\sigma) \ne 4$. 
The ball pair $(\Delta_J\times\Delta_\sigma , P_{\sigma, J})$ is unknotted. 
\end{proposition}

\begin{proof}
First we check that the pair $(\Delta_J\times\Delta_\sigma , P_{\sigma, J})$ is locally flat. Consider the map
$$\phi\colon \Delta_J\times\Delta_\sigma \to \C,\quad (z_0,\dots z_n) \mapsto z_0+\dots +z_n.
$$
By an argument similar to the proof of \cite[Proposition 5]{KZ} one concludes that for small $|t|\leq \epsilon$ the fiber $\phi^{-1}(t)$ is a closed ball. In particular, $\phi^{-1}(0)=P_{\sigma, J}$ and the map $\phi$ exhibits the product structure in a small neighborhood of $P_{\sigma, J}\subseteq \Delta_J\times\Delta_\sigma$. 
%It is obviously true in the interior of $\Delta_J\times\Delta_\sigma$ since $P_{\sigma, J}$ is a smooth submanifold of $\Delta_J\times\Delta_\sigma$. Let $\Delta_{J'}\times\Delta_{\sigma'}$ be a boundary stratum of $\Delta_J\times\Delta_\sigma$ whose intersection with $P_{\sigma, J}$ is non-empty. This means that $\sigma'$ still divides $J'$. 
%LOCAL FLATNESS
%\cite[Proposition 5]{KZ}
As a consequence we also have the local flatness of the sphere pair $(\partial(\Delta_J\times\Delta_\sigma) ,\partial P_{\sigma, J})$. 

Combining Corollary~\ref{cor:complement} and part (1) of Proposition \ref{prop:unknotted} we conclude that the sphere pair $(\partial(\Delta_J\times\Delta_\sigma) ,\partial P_{\sigma, J})$ is unknotted except  the case $\dim (\Delta_J\times\Delta_\sigma) =5$. Then Lemma~\ref{lemma:complement} together with part (2) of Proposition \ref{prop:unknotted} yields the result when $\dim (\Delta_J\times\Delta_\sigma) \ne 4,5$. Lemma~\ref{lemma:42} takes care of the case $\dim (\Delta_J\times\Delta_\sigma) =4$. Thus we are done once we can apply part (2) of Proposition \ref{prop:unknotted} to the case $\dim (\Delta_J\times\Delta_\sigma) =5$, that is, once we know that the sphere pair $(\partial(\Delta_J\times\Delta_\sigma) ,\partial P_{\sigma, J})$ is unknotted in this case.

Let $\dim (\Delta_J\times\Delta_\sigma) =5$. Each face $F$ in $\partial(\Delta_J\times\Delta_\sigma)$ contains $P_{\sigma, J}\cap F$ as an unknotted ball by Lemma~\ref{lemma:42}). On the other hand, the same face $F$ contains the unknotted ball $\mathcal P_{\sigma, J}\cap F$ of the ober-tropical pants. The face posets of unknotted ball pairs $\{(F, P_{\sigma, J}\cap F)\}$ and $\{(F, \mathcal P_{\sigma, J}\cap F)\}$ are identical and the sphere pair $(\partial(\Delta_J\times\Delta_\sigma) ,\partial \mathcal P_{\sigma, J})$ is unknotted by Lemma~\ref{lemma:53}, hence so is the pair $(\partial(\Delta_J\times\Delta_\sigma) ,\partial P_{\sigma, J})$.
\end{proof}

\begin{remark}
Unfortunately the map $\phi\colon \Delta_J\times\Delta_\sigma \to \C$ above is not a fiber bundle at the boundary of its image. Namely the fibers $\phi^{-1}(t)$ for $|t|=1$ are generally lower dimensional. For example, in the $(B^4,B^2)$-case (cf.~Lemma~\ref{lemma:42}) the preimage of the entire circle $|t|=1$ is the M\"obius band. If the fibers were equi-dimensional balls we would have an explicit product structure and, hence, the unknotness for free.  
\end{remark}

%No inductive argument is needed here, contrary to the ober-tropical case below. Combining the Lemma~\ref{lemma:42} and Proposition~\ref{prop:complex_balls} we get the unknotness for all pairs $(\Delta_J\times\Delta_\sigma , P_{\sigma, J})$.}

\begin{proof}[Proof of Theorem \ref{thm:main}]
Propositions \ref{prop:complex_balls} and \ref{prop:ober_balls} provide homeomorphisms of ball pairs $(\Delta_J\times\Delta_\sigma , P_{\sigma, J})$ and $(\Delta_J\times\Delta_\sigma , \pp_{\sigma, J})$ which respect the stratifications. By the Alexander trick \cite{Al23}, the homeomorphisms are homotopically equivalent to the identity on all cells. An inductive application of Proposition~\ref{prop:unknotted} part (4) on strata gives an isotopy which respects the stratification.
\end{proof}

\end{document}